\documentclass[12pt]{amsart}

\usepackage{mathrsfs}
\usepackage{amssymb}

\usepackage[colorlinks,
linkcolor=black,
anchorcolor=blue,
citecolor=black]{hyperref}

\newcommand{\tr}{\mbox{tr}}

\newtheorem{theorem}{Theorem}[section]
\newtheorem{lemma}[theorem]{Lemma}
\newtheorem{proposition}[theorem]{Proposition}
\newtheorem{corollary}[theorem]{Corollary}

 \theoremstyle{definition}
\newtheorem{definition}[theorem]{Definition}
\newtheorem{example}[theorem]{Example}

\theoremstyle{remark}
\newtheorem{remark}[theorem]{Remark}

\numberwithin{equation}{section}

\begin{document}
	\setlength{\baselineskip}{1.2\baselineskip}
		
	\title[On Neumann problems for elliptic and parabolic equations]{On Neumann problems for elliptic and parabolic equations on bounded manifolds}
	
	\author{Sheng Guo}
	
	\address{Department of Mathematics, Ohio State University, Columbus, OH 43210}
	\email{guo.647@osu.edu}
	
	\maketitle
	
\begin{abstract}
	In this paper, we study fully nonlinear second-order elliptic and parabolic equations with Neumann boundary conditions on compact Riemannian manifolds with smooth boundary. 
	We derive oscillation bounds for admissible solutions with Neumann boundary condition $u_\nu = \phi(x)$ assuming the existence of suitable $\mathcal{C}$-subsolutions. 
	We use a parabolic approach to derive a solution of a $k$-Hessian equation with  Neumann boundary condition $u_\nu = \phi(x)$ under suitable assumptions.
\end{abstract}

\section{Introduction}
Let $(\overline{M}, g)$ be a compact Riemannian manifold of dimension $n$ with smooth boundary $\partial M$, and $\chi$ a (0, 2)-tensor on $\overline{M}$. 
In this paper, we consider the following Neumann boundary problem of fully nonlinear second-order elliptic equation of the form
\begin{equation} \label{main_eq_elliptic}
\left\{
\begin{aligned}
&F(\chi_{ij} + u_{ij}) = \psi(x) \qquad &\text{in}\; & M, \\ 
&u_\nu = \phi(x, u) \qquad &\text{on} \; & \partial M, 
\end{aligned}
\right.
\end{equation}
where $\nu$ denotes the unit inner normal vector of $\partial M$.
When $M$ is an open bounded domain in $\mathbb{R}^n$, $\chi \equiv 0$, and 
\begin{equation}\label{phi_z_strictly_+}
\phi_z (x, z) \geq \gamma_0 > 0,
\end{equation}
the Neumann boundary problem (\ref{main_eq_elliptic}) has been actively studied by many researchers in recent decades.
Lieberman-Trudinger \cite{LiebTru86} studied $C^{2 + \alpha}$ regularity of uniformly elliptic equations. 
The celebrated paper \cite{LionsTrudUrbas86} by Lions-Trudinger-Urbas studied Monge-Amp\`ere equations on uniformly convex domains, followed by Wang \cite{Wang92}, Urbas \cite{Urbas98} and Li \cite{S.Li99} on oblique boundary problems\footnote{$u_\beta = \phi(x, u)$ on $\partial \Omega$, where $\beta\cdot\nu \geq \beta_0 >0.$}. 
Schnürer-Smoczyk \cite{SchnurerSmoczyk03} used a parabolic approach to study Monge-Amp\`ere type equations.
Li \cite{S.Li94} studied complex Monge-Amp\`ere equations on bounded strictly pseudocovex domains in $\mathbb{C}^n$. 
Trudinger \cite{Tru87} studied more general fully nonlinear elliptic equations on the unit ball $M = \mathbb{D}^{n}$. 
Urbas \cite{Urbas95} \cite{Urbas96} studied oblique boundary problems for nonuniformly elliptic Hessian equations and curvature equations on uniformly convex domains in dimension 2.
Recently, Ma-Qiu \cite{MaQiu19} studied the $k$-Hessian equations on uniformly convex domains. Guan-Xiang \cite{GuanXiang18} studied general fully nonlinear elliptic equations on compact Riemannian manifolds with smooth boundary. 

In this paper, we mainly focus on the Neumann boundary problem (\ref{main_eq_elliptic}) under the assumption 
\begin{equation}\label{phi_z=0}
	\phi_z(x, z) \equiv 0.
\end{equation}
Comparing to the assumption (\ref{phi_z_strictly_+}), the difficulty for the assumption (\ref{phi_z=0}) is $C^0$ estimates. 
The assumption (\ref{phi_z_strictly_+}) guarantees $C^0$ estimates (see \cite{Tru87}).
However, the assumption (\ref{phi_z=0}) does not imply classical $C^0$ estimates because if $u$ is a solution, then $u + C$ is also a solution for any constant $C$. 
Moreover, there are some obstructions to the existence of solutions under the assumption (\ref{phi_z=0}). For example, consider the $k$-Hessian equations, that is, $F(A) := \sigma_k^{1/k} \big(\lambda(A)\big)$, for $1 \leq k \leq n$. By Maclaurin's inequality and Stokes' theorem, 
\begin{equation}
\begin{aligned}
0 &< \int_{\overline{M}} \psi dV = 	\int_{\overline{M}} \sigma_k^{1/k}\big(\lambda_g (\chi + \nabla^2 u)\big) dV	\\
&\leq c(n,k) \int_{\overline{M}} \tr_g \chi + \Delta_g u \,dV 
= c(n,k) \left( \int_{\overline{M}} \tr_g \chi \,dV - \int_{\partial M} \phi \,dS \right),
\end{aligned}
\end{equation}
that is, 
\begin{equation}
	\int_{\partial M} \phi \,dS < \int_{\overline{M}} \tr_g \chi \,dV,
\end{equation}
\begin{equation}
	\int_{\overline{M}} \psi dV \leq c(n,k) \left( \int_{\overline{M}} \tr_g \chi \,dV - \int_{\partial M} \phi \,dS \right),
\end{equation}
where $\lambda_g$ is defined as eigenvalues with respect to the metric $g$, and $\tr_g \chi$ is the trace of $\chi$ with respect to $g$, and $\Delta_g$ is the Laplace–Beltrami operator, and $\sigma_k$ is the $k$-th elementary symmetric polynomial.
Therefore, compatibility conditions for $\psi(x)$ and $\phi(x)$ are needed. 
Motivated by Sz\'ekelyhidi \cite{Sze15}, we
assume the existence of $\mathcal{C}$-subsolutions $\underline{u}$ with $\underline{u}_\nu = \phi(x)$ on $\partial M$, which is useful to derive oscillation bounds for solutions as well as providing a compatibility condition for $\phi(x)$.
Moreover, we rescale $\psi$ if necessary, for example, $e^c \psi$ or $\psi + c$.
We use a parabolic approach to derive a classical solution for the Neumann boundary problem (\ref{main_eq_elliptic}) under the assumption (\ref{phi_z=0}). 

Due to technical difficulties of second-order estimates, we focus on the existence theorem for $k$-Hessian equations on $\overline{M}$ with certain curvatures assumptions of $\overline{M}$ as follows, which is motivated by Ma-Qiu \cite{MaQiu19}.
\begin{theorem}\label{elliptic_Neumann_sigma_k}
	Let $(\overline{M}, g)$ be a compact Riemanian manifold with non-negative sectional curvatures and uniformly strictly convex boundary $\partial M$ and $\chi$ a smooth (0, 2)-tensor on $\overline{M}$.
	Suppose $\phi(x) \in C^{\infty} (\partial M)$, and there exists $\underline{u} \in C^{\infty}(\overline{M})$ such that $\lambda_g(\chi + \nabla^2 \underline{u}) \in \Gamma_k(\mathbb{R}^n)$\footnotemark, and $\underline{u}_\nu = \phi(x)$.
	Then for any $\psi \in C^{\infty}( \overline{M} )$, there exists a constant $c$ such that the Neumann boundary problem
	\begin{equation}\label{univ_conv_sigma_k_eq_intro}
	\left\{
	\begin{aligned}
	&\sigma_k\big(\lambda_g(\chi_{ij} + u_{ij})\big) = e^{\psi(x) + c} \qquad &\text{in}&\; M, \\ 
	&u_\nu = \phi(x) \qquad &\text{on}& \; \partial M,   
	\end{aligned}
	\right.
	\end{equation}
	has a unique smooth solution $u \in C^{\infty} (\overline{M})$ up to a constant.
	\footnotetext{$\Gamma_k (\mathbb{R}^n) := \{\lambda\in\mathbb{R}^n : \sigma_l(\lambda)>0, \text{for}\; 1\leq l \leq k\}$.}
\end{theorem}

To formulate appropriate conditions on $F$, let $\mathcal{S}^{n\times n}$ be the real-valued-$(n \times n)$-symmetric-matrix-space.
Suppose $F(A)$ is a $C^2$ function defined on an open convex cone $\Gamma \subset \mathcal{S}^{n\times n}$ with vertex at the origin.
Denote
$$F^{ij}(A) := \frac{\partial F}{\partial A_{ij}} (A), \qquad F^{ij, kl}(A) := \frac{\partial^2 F}{\partial A_{ij} \partial A_{kl}} (A).$$
We define
$$\Gamma_k := \big\{A\in\mathcal{S}^{n\times n} \,  : \sigma_l\big(\lambda(A)\big)>0, \text{for}\; 1\leq l \leq k\},$$
where $\sigma_k$ is the $k$-th elementary symmetric polynomial for $1\leq k \leq n$, and $\lambda(A)$ are eigenvalues of $A$.
We may assume:
\begin{align}
& \quad \Gamma_n \subset \Gamma \subset \Gamma_1;  \label{cone}\\
& \quad \text{Ellipticity:}\, \big(F^{ij}(A)\big)_{n\times n} >0, \, \forall A\in\Gamma; \label{ellipticity}\\
& \quad \text{Concavity:}\, \sum_{i,j} F^{ij}(A) \big( B_{ij} - A_{ij} \big) \geq F(B) - F(A), \quad \forall A,B \in \Gamma. \label{concavity}
\end{align}

$F(\chi_{ij} + u_{ij})$ can be defined locally under a local (orthonormal) frame $\{e_i\}$.
It can be defined globally if it is independent of the choice of the local (orthonormal) frame $\{e_i\}$. For example, functions of eigenvalues $F(U_{ij}) := f\big(\lambda_g(U_{ij})\big)$ or linear functions $F(U_{ij}) := A^{ij}U_{ij}$ are well defined globally on $M$.
We call $u \in C^2 (M)$ (or $u \in C^{2, 1} \big(M \times (0, T)\big)$) an admissible solution for the elliptic equation $F(\chi_{ij} + u_{ij}) = \psi$ (or the parabolic equation $F(\chi_{ij} + u_{ij}) = u_t + \psi$ resp.) if $(\chi_{ij} + u_{ij}) \in \Gamma$.

\begin{definition}\label{Def_C-sub}
	We say that $\underline{u} \in C^2(M)$ is a $\mathcal{C}$-subsolution of $F(\chi_{ij} + u_{ij}) = \psi(x)$ if at any $x \in M$, the set
	$$\Big\{ A \in \Gamma : F(A) = \psi(x) \; \text{and}\; A - (\chi_{ij} + \underline{u}_{ij})(x) \in \Gamma_n \Big\}$$
	is bounded.
\end{definition}

\begin{definition}\label{Def_C-sub_parabolic}
	We say that $\underline{u} \in C^{2,1}\big(M\times(0, T)\big)$ is a (parabolic) $\mathcal{C}$-subsolution of $F(\chi_{ij} + u_{ij}) - u_t = \psi(x)$ if at any $(x, t) \in M\times(0,T)$, the set
	$$\Big\{ (A, s) \in \Gamma \times \mathbb{R} : F(A) + s = \psi(x) \; \text{and}\; (A, s) - \big(\chi_{ij} + \underline{u}_{ij}, - \underline{u}_t\big)(x, t) \in \Gamma_n\times(0, \infty) \Big\}$$
	is bounded.
\end{definition}

\begin{remark}
	The definitions above are similar to Sz\'ekelyhidi's \cite{Sze15}. Phong-T\^o \cite{PhongTo17} give a slightly different definition of (parabolic) $\mathcal{C}$-subsolutions. Guo \cite{GuoThesis19} introduce equivalent definitions of (elliptic and parabolic) $\mathcal{C}$-subsolutions for a more general function $F$ which is not necessary to be a function of eigenvalues.
\end{remark}

We consider the following initial-boundary (abbr. IBV) problem with Neumann boundary condition
\begin{equation} \label{main_eq}
\left\{
\begin{aligned}
u_t &= F(\chi_{ij} + u_{ij}) - \psi(x, u, t) \qquad &\text{in}&\; M \times \{t>0\}, \\ 
u_\nu &= \phi(x, u) \qquad &\text{on}& \; \partial M \times \{t \geq 0\},\\
u  &= u_0 \qquad & \text{in}& \; \overline{M}\times\{t = 0\},
\end{aligned}
\right.
\end{equation}
where $\chi$ is a (0, 2)-tensor on $\overline{M}$ and $\nu$ denotes the unit inner normal vector of $\partial M$. In addition, we assume $\phi_z (x,z) \geq 0$.
In this paper, we derive a priori $C^{2+\alpha, 1+\frac{\alpha}{2}}$ estimates for IBV problem (\ref{main_eq}) under suitable assumptions, and obtain the long-time existence results. We derive a solution for the Neumann boundary problem (\ref{main_eq_elliptic}) through a uniform convergence theorem.

The rest of this paper is organized as follows. 
In Section \ref{pre.sec}, we provide some useful formulas and lemmas.
In Section \ref{osc.sec}, we derive a weak Harnack inequality and an Alexandroff-Bakelman-Pucci (A-B-P) type estimate for $u$ satisfying the Neumann boundary condition $u_\nu = 0$ on $\partial M$ and use the approach in \cite{Sze15} to obtain oscillation bounds.
From Section \ref{estimates.sec_1} to Section \ref{estimates.sec_3}, we derive $C^{2, 1}$ a priori estimates for the IBV problem (\ref{main_eq}) under certain assumptions. 
In Section \ref{conclusion.sec}, we obtain a long-time existence theorem for the IBV problem (\ref{main_eq}) under certain assumptions, and use a modified evolution equation to derive a solution of the Neumann boundary problem (\ref{main_eq_elliptic}) under the assumption (\ref{phi_z=0}) for $k$-Hessian equations and linear elliptic equations. 
In Appendix \ref{Harnack.app}, we derive a Harnack inequality for linear parabolic equations with vanishing Neumann boundary conditions.

\vspace{1em}
\noindent
{\bf Acknowledgements.} The author would like to thank his advisor Bo Guan for constructive suggestions and constant support. The author also thanks Barbara Keyfitz, King Yeung Lam for some helpful discussions. The results of this paper are contained in the author’s PhD dissertation at Ohio State University \cite{GuoThesis19}. The author would like to express his gratitude to his Father.

\section{Preliminaries}\label{pre.sec}
For $\lambda  = (\lambda_1, \cdots, \lambda_n) \in \mathbb{R}^n$, 
let 
$$\sigma_k (\lambda) := \sum_{1 \leq i_1 \leq \cdots \leq i_k \leq n} \lambda_{i_1} \cdots \lambda_{i_k} \quad (1 \leq k \leq n),$$
$$\lambda|i := (\lambda_1, \cdots, \lambda_{i-1}, \lambda_{i+1}, \cdots, \lambda_n), \qquad \lambda|ij := (\lambda|i)|j, \qquad \text{etc}.$$

The following are some useful properties for $\sigma_{k}$.
\begin{proposition}
	For $\lambda = (\lambda_1, \cdots, \lambda_n) \in \mathbb{R}^n$, fix $1 \leq k \leq n$, we have
	
	\begin{equation}\label{sigma_k_prop_1}
	\sigma_k (\lambda) = \sigma_k(\lambda|i) + \lambda_i \sigma_{k-1} (\lambda|i), \qquad 1 \leq i \leq n,
	\end{equation}
	\begin{equation}\label{sigma_k_prop_2}
	\sum_i \lambda_i \sigma_{k-1}(\lambda|i) = k \sigma_k(\lambda),
	\end{equation}
	\begin{equation}\label{sigma_k_prop_3}
	\sum_i \sigma_{k-1}(\lambda|i) = (n - k + 1) \sigma_{k-1}(\lambda),
	\end{equation}
	
	If in addition we assume $\lambda \in \Gamma_{k}(\mathbb{R}^n)$ with $\lambda_1 \geq \lambda_2 \geq \cdots \geq \lambda_n$, then we have
	\begin{equation}\label{sigma_k_prop_4}
	\sigma_l (\lambda|i) > 0, \qquad \forall\, 1 \leq l < k \; \text{and}\; 1 \leq i \leq n,
	\end{equation}
	\begin{equation}\label{sigma_k_prop_7}
	\sigma_{k-1}(\lambda|i) \leq \sigma_{k-1}(\lambda|j), \qquad \forall\, \lambda_i \geq  \lambda_j,
	\end{equation}
	\begin{equation}\label{sigma_k_prop_5}
	\lambda_1 \sigma_{k-1}(\lambda|1) \geq \frac{k}{n} \sigma_k (\lambda), 
	\end{equation}
	\begin{equation}\label{sigma_k_prop_8}
	(\text{Maclaurin's inequality}) \quad
	\left(\frac{\sigma_k (\lambda)}{{{n}\choose{k}}}\right)^{\frac{1}{k}} \leq \left(\frac{\sigma_l (\lambda)}{{{n}\choose{l}}}\right)^{\frac{1}{l}}, \qquad \forall\, 1 \leq l \leq k,
	\end{equation}
\end{proposition}
\begin{proof}
	See Chapter 15 Section 4 in \cite{Lieberman96} for details. Proof of (\ref{sigma_k_prop_5}) can be found in Lemma 3.1 in \cite{ChouWang01} or Lemma 2.2 in \cite{HouMaWu10}.
\end{proof}

The following lemma is crucial for the second-order normal-normal estimates.
\begin{lemma}\label{sigma_k_lemma}
	For $\lambda = (\lambda_1, \cdots, \lambda_n) \in \mathbb{R}^n$, 
	denote $f(\lambda) := \sigma_k(\lambda)$, $f_i := \frac{\partial f}{\partial \lambda_i} = \sigma_{k-1}(\lambda | i)$, $\lambda_{min} := \min\{ \lambda_1, \cdots, \lambda_n \}$ and $\lambda_{max} := \max\{\lambda_1, \cdots, \lambda_n\}$. Suppose $\lambda \in \Gamma_{k}(\mathbb{R}^n)$ satisfies $\mu_2 \lambda_{max} \leq \lambda_1 \leq - \mu_1 \lambda_{min}$ with $\mu_1 > 0, 0 < \mu_2 \leq 1$ and $n > k \geq 2$, then 
	\begin{equation}\label{sigma_k_lemma_0}
	f_1 \geq \left( \frac{\mu_2}{\mu_1} \right)^2 \frac{k - 1}{(n-1)(n-2+k) (n-k+1)} \sum_i f_i.
	\end{equation} 
\end{lemma}
\begin{proof}
	The proof is motivated by Ma-Qiu \cite{MaQiu19}.
	
	Assume $\lambda_n = \lambda_{min}$. By (\ref{sigma_k_prop_4}), for any $1 \leq i \leq n$, we have $\lambda|i \in \Gamma_{k-1}(\mathbb{R}^{n-1})$. 
	
	If $\lambda_1 = \lambda_{max}$, then we apply (\ref{sigma_k_prop_5}) to $\lambda|n$ and have
	\begin{equation}
	\lambda_1 \sigma_{k-2}(\lambda|1n) \geq \frac{k-1}{n-1}\sigma_{k-1}(\lambda|n) \geq \mu_2\frac{k-1}{n-1}\sigma_{k-1}(\lambda|n).
	\end{equation}
	If $\lambda_1 < \lambda_{max}$, then we assume $\lambda_2 = \lambda_{max}$. We apply (\ref{sigma_k_prop_7}) and (\ref{sigma_k_prop_5}) to $\lambda|n$ and have
	\begin{equation}
	\begin{aligned}
	\lambda_1 \sigma_{k-2}(\lambda|1n) \geq \mu_2 \lambda_2 \sigma_{k-2}(\lambda|2n) \geq \mu_2\frac{k-1}{n-1}\sigma_{k-1}(\lambda|n).
	\end{aligned}
	\end{equation}
	Therefore, 
	\begin{equation}\label{sigma_k_lemma_1}
	\lambda_1 \sigma_{k-2}(\lambda|1n) \geq \mu_2\frac{k-1}{n-1}\sigma_{k-1}(\lambda|n).
	\end{equation}
	
	{\bf Case 1:} $\sigma_{k-1}(\lambda|1) \geq - \delta \lambda_n \sigma_{k-2}(\lambda|1n)$, where $\delta := \frac{\mu_2}{(n-2+k)\mu_1}$.
	
	Since $\lambda_n < 0$, by (\ref{sigma_k_prop_1}), (\ref{sigma_k_prop_3}) and (\ref{sigma_k_prop_4}), we have
	\begin{equation}\label{sigma_k_lemma_2}
	\sigma_{k-1}(\lambda|n) = \sigma_{k-1}(\lambda) - \lambda_n \sigma_{k-2}(\lambda|n) \geq \sigma_{k-1}(\lambda) = \frac{1}{n-k+1}\sum_i f_i.
	\end{equation}
	By (\ref{sigma_k_lemma_1}) and (\ref{sigma_k_lemma_2}),
	\begin{equation}
	\begin{aligned}
	f_1 & = \sigma_{k-1}(\lambda|1) 
	\geq  \delta (-\lambda_n) \sigma_{k-2}(\lambda|1n) \\
	&\geq  \frac{\delta}{\mu_1} \lambda_1 \sigma_{k-2}(\lambda|1n) \geq  \frac{\delta\mu_2}{\mu_1} \frac{k-1}{n-1}\sigma_{k-1}(\lambda|n) \\
	&\geq \left( \frac{\mu_2}{\mu_1} \right)^2 \frac{k - 1}{(n-1)(n-2+k) (n-k+1)} \sum_i f_i.
	\end{aligned}
	\end{equation}
	
	{\bf Case 2:} $\sigma_{k-1}(\lambda|1) \leq - \delta \lambda_n \sigma_{k-2}(\lambda|1n)$, where $\delta := \frac{\mu_2}{(n-2+k)\mu_1}$.
	
	By (\ref{sigma_k_prop_1}), (\ref{sigma_k_prop_2}), (\ref{sigma_k_prop_4}), (\ref{sigma_k_lemma_1}) and (\ref{sigma_k_lemma_2}), we have
	\begin{equation}\label{sigma_k_lemma_3}
	\begin{aligned}
	k \sigma_k(\lambda|1) & = \sum_{i\neq 1} \lambda_i \sigma_{k-1}(\lambda|1i) 
	= \sum_{i\neq 1} \lambda_i \big(\sigma_{k-1}(\lambda|1) - \lambda_i \sigma_{k-2}(\lambda|1i)\big)\\
	&\leq  \sum_{i\neq 1, \lambda_i \geq 0} \lambda_i \big( - \delta \lambda_n \sigma_{k-2}(\lambda|1n) \big) - \sum_{\lambda_i < 0} \lambda_i^2 \sigma_{k-2}(\lambda|1i)\\
	&\leq (n-2) \delta \lambda_{max} (-\lambda_n) \sigma_{k-2}(\lambda|1n) - \lambda_n^2 \sigma_{k-2}(\lambda|1n)\\
	&\leq \frac{(n-2) \delta \mu_1}{\mu_2} \lambda_n^2 \sigma_{k-2}(\lambda|1n) - \lambda_n^2 \sigma_{k-2}(\lambda|1n)\\
	&= -\frac{k}{(n-2+k)} \lambda_n^2 \sigma_{k-2}(\lambda|1n)
	\leq -\frac{k}{\mu_1^2(n-2+k)} \lambda_1^2 \sigma_{k-2}(\lambda|1n)\\
	&\leq -\frac{\mu_2 (k-1)k}{\mu_1^2 (n-1)(n-2+k)} \lambda_1 \sigma_{k-1}(\lambda|n)\\
	&\leq -\frac{\mu_2 (k-1)k}{\mu_1^2 (n-1)(n-2+k)(n-k+1)} \lambda_1 \sum_i f_i.
	\end{aligned}
	\end{equation}
	By (\ref{sigma_k_prop_1}), (\ref{sigma_k_lemma_3}) and $0<\mu_2 \leq 1$, we have
	\begin{equation}\begin{aligned}
	f_1 \lambda_1 &= \lambda_1 \sigma_{k-1}(\lambda|1) = \sigma_k(\lambda) - \sigma_{k}(\lambda|1)\\
	&\geq \frac{\mu_2 (k-1)}{\mu_1^2 (n-1)(n-2+k)(n-k+1)} \lambda_1 \sum_i f_i \\
	&\geq \frac{\mu_2^2 (k-1)}{\mu_1^2 (n-1)(n-2+k)(n-k+1)} \lambda_1 \sum_i f_i,
	\end{aligned}
	\end{equation}
	which implies (\ref{sigma_k_lemma_0}) since $\lambda_1 > 0$.
\end{proof}

Let $\nu$ be the unit inner normal vector of $\partial M$. The second fundamental form of $\partial M$ can be defined as
\begin{equation}
I\!\!I (X, Y) := -g(\nabla_X \nu, Y) \qquad \forall \; X, Y \in T\partial M.
\end{equation}
Since $\overline{M}$ is compact and $\partial M$ is smooth, then there exists a small $\delta_0 > 0$ such that $\text{dist}(x, \partial M)$ is smooth in $\overline{M_{\delta_0}}:=\{ x\in \overline{M}: d(x) \leq \delta_0 \}$. We can extend $\text{dist}(x, \partial M)$ smoothly to the whole $\overline{M}$ as follows
\begin{equation}\label{dist2bd_function}
d(x) := \left \{ 
\begin{aligned}
\text{dist}(x, \partial M) \qquad & \text{in} \; \overline{M_{\delta_0/2}} := \{ x\in \overline{M}: d(x) \leq \delta_0/2 \},\\
0 < d(x) \leq \delta_0 \qquad & \text{in} \; M.
\end{aligned}\right.
\end{equation}
Then $\nabla d = \nu$ on $\partial M$ and $|\nabla d| = 1$ in $\overline{M_{\delta_0/2}}$.

We list some useful formulas and identities in Riemannian geometry under a local frame $\{e_i\}$ in the following:

\begin{equation}
R^i_{jkl} = \partial_k\Gamma^i_{jl} - \partial_l\Gamma^i_{jk} + \Gamma^i_{km}\Gamma^m_{jl} - \Gamma^i_{lm}\Gamma^m_{jk},
\end{equation} 
\begin{equation}
\text{Sec}(e_i, e_j) = \frac{R_{ijij}}{g_{ii}g_{jj} - g_{ij}^2},
\end{equation}
\begin{equation}
u_{ij} = \partial_j u_i - \Gamma^k_{ij}u_k,
\end{equation}
\begin{equation}
u_{ijk} = \partial_k u_{ij} - \Gamma^m_{ik}u_{mj} - \Gamma^m_{jk}u_{mi},
\end{equation}
\begin{equation}\label{ijk-ikj}
u_{ijk} - u_{ikj} = R^m_{ijk}u_m,
\end{equation}
\begin{equation}
u_{ijkl} - u_{ijlk} = R^m_{jkl} u_{mi} + R^m_{ikl} u_{mj},
\end{equation}
\begin{equation}\label{ijkl-klij}
u_{ijkl} - u_{klij} = R^m_{ijk} u_{ml} + R^m_{ijl} u_{mk} + R^m_{kil} u_{mj} + R^m_{kjl} u_{mi} + R^m_{ijk,l} u_m + R^m_{kil,j} u_m,
\end{equation}
where $u_i := \nabla_i u := \nabla_{e_i} u, u_{ij} := \nabla_j \nabla_i u, u_{ijk} :=\nabla_k \nabla_j \nabla_i u$, etc.

\begin{lemma}\label{lemma_u_ijk}
	\begin{equation}
	\nabla_i\nabla_j\nabla_k u = \nabla_i\nabla_j (\nabla_k u) - \nabla_i\nabla_{\nabla_j e_k}u - \nabla_j\nabla_{\nabla_i e_k}u
	- \nabla_{\nabla_i\nabla_j e_k}u.
	\end{equation}
\end{lemma}
\begin{proof}
	See Lemma 1.3.1 in \cite{GuoThesis19}.
\end{proof}

\section{Oscillation bounds}\label{osc.sec}
In this section, derive an oscillation bound for admissible solutions of the Neumann boundary problem (\ref{main_eq_elliptic}) under the assumption (\ref{phi_z=0}).
Suppose $u$ is such an admissible solution mentioned above, and $\underline{u}$ is a $\mathcal{C}$-subsolution for (\ref{main_eq_elliptic}) with $\underline{u}_\nu = \phi(x)$ on $\partial M$. Let $v := u - \underline{u}$, then $v$ satisfies the following Neumann boundary problem
\begin{equation}
\left\{
\begin{aligned}
&F\big((\chi_{ij} + \underline{u}_{ij}) + v_{ij}\big) = \psi(x) \qquad &\text{in}\;& M, \\ 
&v_\nu = 0 \qquad &\text{on} \;& \partial M,
\end{aligned}
\right.
\end{equation}
with $\underline{v} \equiv 0$ in $\overline{M}$ as a $\mathcal{C}$-subsolution. Since $\text{osc}_{\overline{M}} u \leq \text{osc}_{\overline{M}} v + \text{osc}_{\overline{M}} \underline{u}$, then it suffices to derive an oscillation bound for $v$ as follows.

\begin{theorem}\label{osc_thm}
	Let $(\overline{M}, g)$ be a compact Riemannian manifold with smooth boundary $\partial M$. 
	Suppose $F$ defined on $\Gamma \supset \Gamma_n$ satisfies ellipticity (\ref{ellipticity}). 
	Let $A$ be a smooth strictly positive (2, 0)-tensor on $\overline{M}$.
	Suppose $\underline{u} \equiv 0$ is a $\mathcal{C}$-subsolution of $F(\chi_{ij} + u_{ij}) = \Psi(x)$ in $\overline{M}$. 
	For any $u \in C^2(\overline{M})$ with $u_\nu = 0$ on $\partial M$, if $F(\chi_{ij} + u_{ij}) \leq \Psi(x)$, and $A^{ij}u_{ij} \geq - C_0$ with $C_0 \geq 0$, then there exists a constant $C$ depending only on $F$, $\Psi$, $\chi$, $A$, $C_0$ and the background geometric data such that
	\begin{equation}
	\sup_{\overline{M}} u - \inf_{\overline{M}} u \leq C.
	\end{equation}
\end{theorem}
\begin{remark}
	(1) If $\Gamma \subset \Gamma_1$, we have $\tr_g \chi + \Delta_g u > 0$. If we choose $A = g^{-1}$, then $g^{ij}u_{ij} \geq - \max\{\sup_{\overline{M}} \tr_g \chi, 0\} =: -C_0$.
	
	(2) If $F(U) := A^{ij}U_{ij}$, then $F(\chi_{ij} + u_{ij}) = \psi(x)$ implies $A^{ij}u_{ij} \geq - \max\big\{ - \inf_{\overline{M}} \psi + \sup_{\overline{M}} A^{ij}\chi_{ij}, 0\big\} =: - C_0$.
	
	(3) As in \cite{Sze15}, we can obtain a similar oscillation bound for elliptic equation $F(\chi_{i\bar{j}} + u_{i\bar{j}}) = \psi(x)$ on compact complex manifolds with smooth boundary.
\end{remark}

To adopt the approach of Proposition 11 in \cite{Sze15}, we need to derive a weak Harnack inequality and an Alexandroff-Bakelman-Pucci (A-B-P) type estimate under Neumann boundary condition $u_\nu = 0$ on $\partial M$.

Let $B_R(0)$ be a ball of radius $R$ centered at the origin in $\mathbb{R}^n$, and
$$B_{R}^{+}(0) := B_{R}(0) \cap \{y_n \geq 0\}, \qquad B_{R}^{-}(0) := B_{R}(0) \cap \{y_n < 0\}.$$
Theorem 9.22 in \cite{GT} derives a weak Harnack inequality on $B_{2R}(0)$.
We can similarly derive the following weak Harnack inequality on $B_{2R}^{+}(0)$ under Neumann boundary condition $u_n = 0$ on $B_{2R}^{+}(0) \cap \{y_n = 0\}$.
\begin{theorem}\label{weak_harnack_bd_thm}
	Let $u \in W^{2,n}\big(B_{2R}^{+}(0)\big)$ be a non-negative function satisfying $Lu := a^{ij}u_{ij} + b^{i}u_{i} + cu \leq f$\footnote{$|b|/\mathcal{D}^*, f/\mathcal{D}^* \in L^n(\Omega)$ and $c \leq 0$, where $\mathcal{D}^* := \big(\det(a_{ij})\big)^{1/n}$.} in $B_{2R}^{+}(0)$ and $u_n = 0$ on $B_{2R}^{+}(0) \cap \{y_n = 0\}$, then 
	\begin{equation}\label{w_harnack_ineq_bd}
	\left( \frac{1}{|B_R^{+}|} \int_{B_R^{+}} u^p \right)^{\frac{1}{p}} \leq C \left( \inf_{B_R^{+}} u + \frac{R}{\theta} ||f||_{L^n(B_{2R}^{+})} \right),
	\end{equation}
	where $\theta \delta_{ij}\leq a^{ij} \leq \frac{1}{\theta} \delta_{ij}$, and $p$ and $C$ are positive constants depending only on $n, \theta, |b|R/\theta, |c|R^2/\theta$.
\end{theorem}
\begin{proof}
	For any $y = (y', y_n) \in \mathbb{R}^n$, define $\text{Ref}(y) := (y', - y_n)$.
	For any $y \in B_{2R}^{-}(0)$, let
	$$u(y) := u\big(\text{Ref}(y)\big) \quad (\text{same to $c$ and $f$}),$$ $$\big(a_{ij}\big)(y) := \text{diag} (1, \cdots, 1, -1) \;\, \big(a_{ij}\big)\big(\text{Ref}(y)\big)\;\, \text{diag} (1, \cdots, 1, -1),$$
	$$b_i (y) := b_i\big(\text{Ref}(y)\big) \quad \text{for}\; i<n, \qquad b_n (y) := -b_n\big(\text{Ref}(y)\big).$$
	
	Obviously, $\det(a_{ij})$ is well defined on $B_{2R}(0)$.
	Since $u_n = 0$ on $\{y_n = 0\}$, then $u_{in} = 0$ on $\{y_n = 0\}$ for $i < n$.
	Even though $a_{in}$, $b_n$ for any given $i < n$ may not be continuous on $\{y_n = 0\}$, since $a_{in} u_{in} = 0$ and $b_n u_n =0$, then the extended $Lu$ is well defined on $B_{2R}(0)$. 
	Let $\tilde{u} := u + c$, $w := - \log \tilde{u}$, $\eta (y) := \big(1 - |y|^2\big)^\beta$ and $v := \eta w$ for some constants $c > 0$ and $\beta \geq 1$. Then $Lw$, $L\eta$ and $Lv$ are also well defined on $B_{2R}(0)$.
	We can apply the same proof of Theorem 9.22 in \cite{GT} to prove (\ref{w_harnack_ineq_bd}).
\end{proof}
\begin{remark}
	If $a_{in} = 0$ and $b_n = 0$ on $\{y_n = 0\}$ for $i < n$, then the extended $a_{ij}$ and $b_i$ are continuous in $B_{2R}(0)$. We can apply Theorem 9.22 in \cite{GT} directly to the extended $u$ to prove the weak Harnack inequality (\ref{w_harnack_ineq_bd}).
\end{remark}

Sz\'ekelyhidi \cite{Sze15} derived the following A-B-P type estimates.
\begin{proposition}[\cite{Sze15}, Proposition 10]\label{ABP_max}
	Suppose $v \in C^2 \big(B_1(0)\big)$ satisfy $v(0) + \frac{\epsilon}{2} \leq \inf_{\partial B_1} v$ with $\epsilon > 0$. Define
	\begin{equation}\label{P_set}
	P := \left\{ x\in B_1(0) : 
	\begin{aligned}
	& |D v(x)| < \frac{\epsilon}{2}, \text{and for all} \; y \in B_1(0)\\
	& v(y) \geq v(x) + Dv(x) \cdot (y - x)
	\end{aligned}
	\right\}.
	\end{equation}
	Then
	\begin{enumerate}
		\item for any $x\in P$, we have
		\begin{equation}
		v(x) < v(0) + \frac{\epsilon}{2}, \qquad D^2 v(x) \geq O;
		\end{equation}
		\item there exists a constant $c_0 = c_0(n)$ such that
		\begin{equation}
		c_0 \epsilon^n \leq \int_{P} \det (D^2 v).
		\end{equation}
	\end{enumerate}
\end{proposition}

Denote
\begin{equation}\label{B_1,c}
B_{1, c}^+ := B_1(0) \cap \{ y_n \geq c \} \subset \mathbb{R}^n, \qquad B_{1, c}^- := B_1(0) \cap \{ y_n < c \} \subset \mathbb{R}^n,
\end{equation}
where $-1 \leq c \leq 0$.
We use the above A-B-P type estimates to derive the following A-B-P type estimates on $B_{1,c}^+$ under Neumann boundary condition $u_n = 0$ on $\{ y_n = c\}$ if $c > -1$.

\begin{proposition}\label{B_c,1_prop}
	Let $u \in C^2 (B_{1,c}^+)$, where $B_{1,c}^+$ is defined in (\ref{B_1,c}) for $-1 \leq c \leq 0$. Suppose $u_n = 0$ on $\{ y_n = c\}$ if $c > -1$. Suppose $u$ attains an infimum at the origin $0$, then for any $\epsilon > 0$, there exists a subset $P_c \subset B_{1,c}^+$ such that the following hold
	\begin{enumerate}
		\item for any $y \in P_c$, we have
		\begin{equation}\label{B_c,1_ineq_1}
		u(y) < u(0) + \frac{1}{2}\epsilon, \qquad D^2 u(y) \geq - \epsilon I, \qquad |Du|(y) < 2\epsilon;
		\end{equation} 
		\item there exists a constant $c_0 = c_0(n)$ such that
		\begin{equation}
		c_0 \epsilon^n \leq \int_{P_c} \det (D^2 u + \epsilon I).
		\end{equation}
	\end{enumerate}
\end{proposition}

\begin{proof}	
	When $c > -1$ and $u_n = 0$ on $\{ y_n = c\}$, we have an even extension of $u$ to $B_1(0)$ as follows.
	For $y = (y', y_n) \in \mathbb{R}^n$, let
	$$\text{Ref}_c (y) := (y', 2c - y_n)$$
	be the reflection map about $\{ y_n = c \}$. For any $y \in B_{1,c}^-$, let
	$$u(y) := u\big(\text{Ref}_c (y)\big).$$
	Since $u_n = 0$ on $\{ y_n = c\}$, then the extended $u \in C^2 \big(B_1(0)\big)$ and $u$ still attains an infimum at $0$.
	
	Let $$v(y) := u(y) + \frac{1}{2} \epsilon |y|^2.$$
	Then $v(0) + \frac{1}{2}\epsilon = u(0) + \frac{1}{2}\epsilon \leq \inf_{\partial B_1} v$.
	We define $P$ for $v$ as in (\ref{P_set}), and
	let $$P_c^+ := P \cap B_{1,c}^+, \qquad P_c^- := P \cap B_{1,c}^-, \qquad P_c := P_c^+ \cup \text{Ref}_c (P_c^-)\subset B_{1,c}^+.$$
	For any $y\in P$, by Proposition \ref{ABP_max}, $u(y) \leq v(y) < v(0) + \frac{1}{2}\epsilon = u(0) + \frac{1}{2}\epsilon$, and $|Du|^2(y) \leq 2 \big(|Dv|^2(y) + \epsilon^2 |y|^2\big) < \frac{5}{2} \epsilon^2 $, and $D^2 u(y) + \epsilon I \geq O$, that is, (\ref{B_c,1_ineq_1}) holds. 
	For any $y \in \text{Ref}_c (P_c^-)$, since $u(y) = u\big(\text{Ref}_c (y)\big)$, $|Du|^2(y) = |Du|^2\big(\text{Ref}_c(y)\big)$, and $\lambda\big(D^2 u (y)\big) = \lambda\big(D^2 u (\text{Ref}_c (y))\big)$, where $\lambda$ are the eigenvalues of the Hessian matrix, then (\ref{B_c,1_ineq_1}) also holds, and
	\begin{equation}
	\int_{\text{Ref}_c (P_c^-)} \det (D^2 v) = \int_{P_c^-} \det (D^2 v).
	\end{equation}
	Hence, by Proposition \ref{ABP_max}, 
	\begin{equation}
	c_0 \epsilon^n \leq \int_{P} \det (D^2 v) = \int_{P_c^+} \det (D^2 v) + \int_{\text{Ref}_c (P_c^-)} \det (D^2 v) \leq 2 \int_{P_c} \det (D^2 v).
	\end{equation}
\end{proof}

With Theorem \ref{weak_harnack_bd_thm} and Proposition \ref{B_c,1_prop}, we can modify the proof of Proposition 11 in \cite{Sze15} to prove Theorem \ref{osc_thm}.

\begin{proof}[Proof of Theorem \ref{osc_thm}]
	
	Assume $\sup_{\overline{M}} u =0$, it suffices to prove that 
	\begin{equation}
	L := \inf_{\overline{M}} u \geq - C.
	\end{equation}

	For $x_\alpha \in M$, we can choose a coordinate chart $(U_\alpha, \phi_\alpha)$ such that $x_\alpha \in U_\alpha \subset M$ and $\phi_\alpha (U_\alpha) = B_2(0) \subset \mathbb{R}^n$ with $\phi (x_\alpha) = 0$. For $x_\alpha \in \partial M$, we can choose a coordinate chart $(U_\alpha, \phi_\alpha)$ such that $x_\alpha \in U_\alpha \subset \overline{M}$ and $\phi_\alpha (U_\alpha) = B_2^+ (0)\subset \mathbb{R}^n$ with $\phi_\alpha (x_\alpha) = 0$ and $\phi_\alpha (U_\alpha \cap \partial M) = B_{2}(0) \cap \{y_n = 0\}$, and $\phi_\alpha^* (\frac{\partial}{\partial y_n}) / |\phi_\alpha^* (\frac{\partial}{\partial y_n})| = \nu$ on $\partial M$, where $B_{2}^{+}(0) := B_{2}(0) \cap \{y_n \geq 0\}$. Meanwhile, $u_\nu = 0$ on $\partial M$ implies $u_n = 0$ on $\{y_n = 0\}$ in coordinate charts chosen above. 
	Under each coordinate chart $(U_\alpha, \phi_\alpha)$, denote $$D_i := \frac{\partial}{\partial y_i},\quad \nabla_i := \nabla_{\phi^*_{\alpha}(\frac{\partial}{\partial y_i})},\quad g_{ij} := g\Big(\phi^*_{\alpha}(\frac{\partial}{\partial y_i}), \phi^*_{\alpha}(\frac{\partial}{\partial y_j})\Big),$$ and $\Gamma_{ij}^k$ the christoffel symbols with respect to $g_{ij}$.
	Let $$\tilde{U}_\alpha := \phi_\alpha^{-1} \big(B_1(0)\big) \quad \text{or} \quad \tilde{U}_\alpha := \phi_\alpha^{-1} \big(B_1^+(0)\big).$$ Since $\cup_{x_\alpha \in \overline{M}} \tilde{U}_\alpha$ is an open cover of compact manifold $\overline{M}$, we can find a finite open cover $\cup_{\alpha \in I} \tilde{U}_\alpha$ with $I$ a finite index set.

	Since $A^{ij}u_{ij} \geq - C_0$ in $\overline{M}$, that is, $A^{ij} D_i D_j (-u) - A^{ij}\Gamma^k_{ij} D_k(-u) \leq C_0$ in local coordinates, then by weak Harnack inequality (see Theorem 9.22 in \cite{GT} and Theorem \ref{weak_harnack_bd_thm}), 
	\begin{equation}\label{osc_Lp_loc}
	\left( \int_{\tilde{U}_\alpha} (-u)^p\, d\mu_{\overline{M}} \right)^{\frac{1}{p}} \leq C_1 \left[ \inf_{\tilde{U}_\alpha} (-u) + 1 \right],
	\end{equation}
	where $p, C_1$ depends only on the finite covering $\{(U_\alpha, \phi_\alpha)\}_{\alpha\in I}$, $A$, $C_0$, and the background geometric data. For any $\alpha, \beta \in I$ with $U_\beta \cap U_\alpha \ne \emptyset$, we have
	\begin{equation}
	\begin{aligned}
	\inf_{\tilde{U}_\beta} (-u)\, \big[\text{Vol} (\tilde{U}_\alpha \cap \tilde{U}_\beta)\big]^{\frac{1}{p}} 
	&\leq \left( \int_{\tilde{U}_\alpha \cap \tilde{U}\beta} (-u)^p\, d\mu_{\overline{M}} \right)^{\frac{1}{p}} \leq \left( \int_{\tilde{U}_\alpha} (-u)^p\, d\mu_{\overline{M}} \right)^{\frac{1}{p}} \\
	&\leq C_1 \left[ \inf_{\tilde{U}_\alpha} (-u) + 1 \right].
	\end{aligned}
	\end{equation}
	Since $\sup_{\overline{M}} u = 0$, then there exists $\tilde{U}_{\alpha_0}$ with $\alpha_0 \in I$ such that $\inf_{\tilde{U}_{\alpha_0}} (-u) = 0$. For $\alpha_1 \in I$ such that $\tilde{U}_{\alpha_1} \cap \tilde{U}_{\alpha_0} \neq \emptyset$, we have $\inf_{\tilde{U}_{\alpha_0}} (-u) \leq C_1 / \big[\text{Vol} (\tilde{U}_{\alpha_0} \cap \tilde{U}_{\alpha_1})\big]^{\frac{1}{p}}$. By induction, for any $\alpha \in I$, we have $\inf_{\tilde{U}_\alpha} (-u) \leq C_2$, where $C_2$ depends only on $C_1, \min_{\alpha, \beta \in I, \tilde{U}_\alpha \cap \tilde{U}_\beta\neq\emptyset}[\text{Vol} (\tilde{U}_\alpha \cap \tilde{U}_\beta)]^{\frac{1}{p}}, |I|$. Hence, by (\ref{osc_Lp_loc}), we obtain a global $L^p$ bound for $u$ 
	\begin{equation}\label{osc_Lp_global}
	||u||_{L^p(\overline{M})} \leq C_3,
	\end{equation}
	where $p, C_3$ depends only on the finite covering $\{(U_\alpha, \phi_\alpha)\}_{\alpha\in I}$, $A$, $C_0$, and the background geometric data.
	
	Suppose $u$ attains an infimum $L$ at $p\in \overline{M}$, then there exists $\alpha \in I$ such that $p \in \tilde{U}_\alpha$.
	Obviously, $B_1\big(\phi_\alpha(p)\big) \subset B_2(0)$.
	Let $T (y) := y - \phi_\alpha (p)$ be a translation in $\mathbb{R}^n$, 
	and $$\tilde{U_p} := \phi_\alpha^{-1} \Big(\phi_\alpha(U_\alpha) \cap B_1\big(\phi_\alpha(p)\big)\Big), \qquad \phi_p := T\circ\phi_\alpha,$$
	then $(\tilde{U_p}, \phi_p)$ is a coordinate chart near $p$. 
	Let 
	\begin{equation}
	c:=\left\{
	\begin{aligned}
	&-1, \qquad &\text{if}\;& \phi_\alpha(U_\alpha) = B_2(0),\\
	&- y_n \big(\phi_\alpha (p)\big), \qquad &\text{if}\;& \phi_\alpha(U_\alpha) = B^+_2(0).
	\end{aligned}\right.
	\end{equation}
	Obviously, $\phi_p (p) = 0$ and $\phi_p (\tilde{U_p}) = B_{1,c}^+$, and $u_n = 0$ on $\{y_n = c\} \cap B_{1,c}^+$ if $c > -1$, where $B_{1,c}^+$ is defined in (\ref{B_1,c}). By Proposition \ref{B_c,1_prop}, for any $\epsilon > 0$, there exists a subset $P_c \subset B_{1,c}^+$ such that (i) for any $y \in P_c$, $u(y) < L + \frac{1}{2}\epsilon$, and $D^2 u(y) \geq - \epsilon I$, and $|Du|(y) < 2\epsilon$; (ii) there exists a constant $c_0 = c_0(n)$ such that
	\begin{equation}\label{P_c_ineq_2}
	c_0 \epsilon^n \leq \int_{P_c} \det (D^2 u + \epsilon I).
	\end{equation}
	When $L < - \frac{\epsilon}{2}$, for any $y \in P_c$, we have $-u(y) > - L - \frac{1}{2}\epsilon > 0$, and hence, by (\ref{osc_Lp_global}),
	\begin{equation}\label{vol_up_bd}
	\left|L + \frac{\epsilon}{2}\right|^p \, \text{Vol}\big(\phi_p^{-1}(P_c)\big) \leq ||-u||^p_{L^p(\overline{M})} \leq C_3^p.
	\end{equation}
	Moreover, we can find universal constants $\theta$ and $C_4$ depending on the finite cover $\{(U_\alpha, \phi_\alpha)\}_{\alpha\in I}$ and the background geometric data such that $(n\times n)$-matrices $(g_{ij})$ and $\big(\Gamma_{ij}^k\big)$ (fix $k$) satisfy $\theta I \leq (g_{ij}) \leq \frac{1}{\theta}I$ with $0 < \theta \leq 1$, and $- C_4 I \leq \big(\Gamma_{ij}^k\big) \leq C_4 I$ for all $k$. For any $y \in P_c$, since $\nabla_i\nabla_j u = D_i D_j u - \Gamma_{ij}^k D_k u$, and $D^2 u(y) \geq - \epsilon I$, and $|D_i u|(y) < 2\epsilon$, then
	\begin{equation}\label{osc_ineq_1}
	\nabla^2 u (y) > - C_5\epsilon g (y),
	\end{equation} 
	where $C_5$ depends only on $C_4, \theta, n$.
	Suppose $\{E_\alpha\} \subset T\overline{M}$ is a local orthonormal frame near $y$.
	Since $\underline{u} \equiv 0$ is a $\mathcal{C}$-subsolution of $F(\chi_{\alpha\beta} + u_{\alpha\beta}) = \Psi$, then we \textbf{claim} that there exists a constant $\delta > 0$, $C_6 > 0$ depending only on $F, \chi, \Psi$ and background geometric data such that
	\begin{equation}\label{osc_ineq_2}
	\Big((\chi_{\alpha\beta})(y) - \delta I + \Gamma_n \Big) \cap \{ (A_{\alpha\beta}) \in \Gamma: F(A_{\alpha\beta}) \leq \Psi(y) \} \subset \{A: ||A|| \leq C_6\}\footnotemark.
	\end{equation}
	Since $F(\chi_{\alpha\beta} + u_{\alpha\beta})(y) \leq \Psi(y)$, if we choose $\epsilon = \delta/C_5$, then by (\ref{osc_ineq_1}) and (\ref{osc_ineq_2}),
	\begin{equation}
	\big(\chi_{\alpha\beta} + u_{\alpha\beta}\big) (y) \in \{A: ||A|| \leq C_6\},
	\end{equation}
	which implies $||\nabla^2 u||_g(y) \leq C_7$, where $C_7$ depends only on $C_6, \sup_{\overline{M}}||\chi||_g$. Since $D_i D_j u  = \nabla_i\nabla_j u + \Gamma_{ij}^k D_k u$, then $(D_i D_j u) \leq C_8 I$, where $C_8$ depends only on $C_4$, $C_7$, $\theta$, $\epsilon$ and $n$. 
	Hence,	by (\ref{P_c_ineq_2}) and (\ref{vol_up_bd}), 
	\begin{equation}
	c_0 \epsilon^n \leq (C_8 + \epsilon)^n \text{Vol}(P_c) \leq \frac{(C_8 + \epsilon)^n}{\theta^{n/2}} \text{Vol}\big(\phi_p^{-1}(P_c)\big) \leq \frac{C_9}{\left|L + \frac{\epsilon}{2}\right|^p},
	\end{equation}
	which implies $L \geq - C_{10}$.
	
	We finish the proof by proving the claim above. For any $x \in \overline{M}$, 
	suppose $\{E_\alpha\} \subset T\overline{M}$ is a local orthonormal frame near $x$, and
	let $\Delta^\delta(x) := \Big((\chi_{\alpha\beta})(x) - \delta I + \Gamma_n \Big) \cap \{ (A_{\alpha\beta}) \in \Gamma: F(A_{\alpha\beta}) \leq \Psi(x) \}$.
	Obviously, $\Delta^{\delta_1}(x) \subset \Delta^{\delta_2}(x)$ for $0 < \delta_1 \leq \delta_2$.
	If $F(\chi_{\alpha\beta})(x) > \Psi(x)$, then there exists $\delta_x > 0$ such that $F(\chi_{\alpha\beta} - \delta_x\,\delta_{\alpha\beta}) > \Psi$ near $x$. 
	Hence, $\Delta^{\delta_x}(y) = \emptyset$ for $y$ in a neighborhood $U_x$ of $x$.
	If $F(\chi_{\alpha\beta})(x) \leq \Psi(x)$, by continuity and ellipticity of $F$, there exists $\delta_x > 0$ such that $\mathcal{C}\Big((\chi_{\alpha\beta}) - \delta_x\,I; (\chi_{\alpha\beta}) - \delta_x\,I, \Psi\Big) = \overline{\Gamma_n}$ near $x$, which implies that there exists a constant $T_x > 0$ such that for any $(P_{\alpha\beta}) \in \overline{\Gamma_n} \cap \{A: ||A|| = 1\}$, we have $F(\chi_{\alpha\beta} - \delta_x\,\delta_{\alpha\beta} + T_x P_{\alpha\beta}) > \Psi$ near $x$. 
	Hence, the $\Delta^{\delta_x}(y) \subset \{ A: ||A|| \leq \sup_{\overline{M}}||\chi||_g + \sqrt{n}\delta_x + T_x \}$ for $y$ in a neighborhood $U_x$ of $x$. 
	Since $\overline{M}$ is compact, then there is a finite open cover $\{U_{x_i}\}_{i=1}^{N}$ over $\overline{M}$. 
	We choose $\delta := \min_{1\leq i \leq N} \{\delta_{x_i}\}$ and $C_6 := \max_{1\leq i\leq N}\{ \sup_{\overline{M}}||\chi||_g + \sqrt{n}\delta_{x_i} + T_{x_i} \}$.
\end{proof}

\section{A priori estimates}\label{estimates.sec_1}

In this section, we assume $F \in C^2(\Gamma)$ satisfies (\ref{cone}) (\ref{ellipticity}) (\ref{concavity}) and the following two additional conditions:
\begin{align}
& \quad  F(A) = f\big(\lambda(A)\big),  \text{where}\; f \; \text{is symmetric}; \label{eigenv} \\
& \quad \text{for any $\sigma < \sup_{\Gamma} F$ and $A \in \Gamma$, we have $\lim_{t\rightarrow \infty} F(tA) > \sigma$\footnote.}. \label{extracondition_1}
\end{align}
\footnotetext{By Lemma 9 in \cite{Sze15}, $\sum_i F^{ii} > \tau$, where $\tau$ depending only on the level set of $F$.}
We assume $u$ is an admissible solution of the IBV problem (\ref{main_eq}).
We derive the following a priori estimates $\sup_{\overline{M_T}}|u_t|$, $\sup_{\overline{M_T}}|u|$, $\sup_{\overline{M_T}}|\nabla u|$ and $\sup_{\overline{M_T}}|\nabla^2 u|$, where $M_T := M \times (0, T]$ with parabolic boundary $\partial M_T := \overline{M_T} \setminus M_T$.

We have the following maximum principle for $u_t$.
\begin{lemma}\label{u_t_bounds}
	Suppose $u_t \in C^{2, 1}\big(\overline{M}\times (0, T]\big) \cap C^0(\overline{M_T})$. Suppose $\phi_z (x, z) \geq 0$ on $\partial M$.
	
	(1) When $\psi_z (x, z, t) \equiv 0$, we have
	\begin{equation}
	\sup_{\overline{M_T}}|u_t| \leq \sup_{\overline{M}\times \{t=0\}}  |u_t| + t\sup_{\overline{M_T}}|\psi_t|.
	\end{equation}
	Especially, when $\phi_z(x,z) \equiv 0$, we have
	\begin{equation}
	\inf_{\overline{M}\times \{t=0\}} u_t - t\sup_{\overline{M_T}} |\psi_t| \leq u_t \leq \sup_{\overline{M}\times \{t=0\}}  u_t + t\sup_{\overline{M_T}}|\psi_t|, \qquad \forall\,(x, t)\in \overline{M_T}.
	\end{equation}
	
	(2) When $\psi_t (x, z, t) \equiv 0$, we have
	\begin{equation}
	\min\Big\{\inf_{\overline{M}\times \{t=0\}} u_t,\, 0\Big\} \leq e^{\lambda t} u_t \leq \max\Big\{\sup_{\overline{M}\times \{t=0\}} u_t,\, 0\Big\}, \qquad \forall\,(x, t)\in \overline{M_T},
	\end{equation}
	where $\lambda := \inf_{\overline{M_T}}\psi_z(x, u)$.
\end{lemma}
\begin{remark}
	When $\phi = \phi(x)$ on $\partial M$ and $\psi = \psi (x)$ in $\overline{M_T}$, then $\sup_{\overline{M}\times \{t=t_1\}}u_t \geq \sup_{\overline{M}\times \{t=t_2\}}u_t$ and $\inf_{\overline{M}\times \{t=t_1\}}u_t \leq \inf_{\overline{M}\times \{t=t_2\}}u_t$ for $0 \leq t_1 \leq t_2$. 
\end{remark}

\begin{proof}
	Let $\mathcal{L}[u]$ be a linear parabolic operator on $M_T$, which is locally written as $\mathcal{L}[u] := F^{ij}\nabla_{i}\nabla_j - \partial_t$, where $F^{ij} := F^{ij}(\chi_{ij} + u_{ij})$.
	Differentiating (\ref{main_eq}) about $t$, we have
	\begin{equation}\label{u_tt}
	\mathcal{L}[u] u_t = (\psi)_t.
	\end{equation}
	
	(1) When $\psi_z (x, z, t) \equiv 0$.
	
	For any $\epsilon > 0$, let $v^\epsilon (x,t):= u_t(x,t) + t\sup_{\overline{M_T}}|\psi_t| + \epsilon t$. Then in $M_T$, by (\ref{u_tt}),
	\begin{equation}\label{u_t_lower}
	\mathcal{L}[u] v^\epsilon = \mathcal{L}[u] u_t - \sup_{\overline{M_T}}|\psi_t| -\epsilon \leq - \epsilon.
	\end{equation}

	Suppose $v^\epsilon$ attains a minimum at $(x_0, t_0) \in M_T$, then at $(x_0, t_0)$, $\mathcal{L}[u] v^\epsilon \geq 0$, which contradicts to (\ref{u_t_lower}).
	
	Suppose $v^\epsilon$ attains a minimum at $(x_0, t_0) \in \partial M\times (0,T]$. Since $\mathcal{L}[u] v^\epsilon \leq -\epsilon$ in $M_T$, by strong maximum principle, $v^\epsilon_\nu(x_0, t_0) > 0$. 
	Since $ v_\nu^\epsilon =  u_{t\nu} = \phi_z u_t$ and $\phi_z \geq 0$ on $\partial M$, then $\phi_z \big(x_0, u(x_0, t_0)\big) > 0$ and $u_t(x_0, t_0) > 0$. Hence, we have $v^\epsilon \geq v^\epsilon(x_0, t_0) > 0$ in $\overline{M_T}$, that is, 
	\begin{equation}\label{u_t_est_ineq_1}
	u_t > - t \sup_{\overline{M_T}}|\psi_t| - \epsilon t.
	\end{equation}
	
	Suppose $v^\epsilon$ attains a minimum at $(x_0, t_0) \in M\times \{t = 0\}$, then in $\overline{M_T}$,
	\begin{equation}\label{u_t_est_ineq_2}
	u_t \geq \inf_{\overline{M}\times \{t=0\}} u_t - t\sup_{\overline{M_T}} |\psi_t| - \epsilon t.
	\end{equation}
	
	Therefore, for any $(x, t) \in \overline{M_T}$, we have
	\begin{equation}
	u_t \geq \min\Big\{\inf_{\overline{M}\times \{t=0\}} u_t,\, 0\Big\} - t\sup_{\overline{M_T}} |\psi_t| - \epsilon t.
	\end{equation}
	When $\phi_z \equiv 0$ on $\partial M$, (\ref{u_t_est_ineq_1}) does not hold because it requires $\phi_z > 0$ at $(x_0, t_0)$, and consequently, 
	\begin{equation}
	u_t \geq \inf_{\overline{M}\times \{t=0\}} u_t - t\sup_{\overline{M_T}} |\psi_t| - \epsilon t.
	\end{equation}
	We let $\epsilon$ tends to 0, then we obtain the lower bound for $u_t$.
	
	Similarly, we can obtain the upper bound for $u_t$.

	(2) When $\psi_t (x, z, t) \equiv 0$.
	
	Motivated by \cite{SchnurerSmoczyk03}, for any $\epsilon > 0$, let 
	$w^\epsilon (x, t) := e^{2\lambda t} (u_t^+)^2 (x, t) - \epsilon t$, where $u_t^+ := \max \{u_t, 0\}$.
	Suppose $w^\epsilon$ attains a maximum at $(x_0, t_0) \in \overline{M_T}$.
	At $(x_0, t_0)$, there are two possibilities
	\footnote{If $(u_t^+)^2(x_0, t_0) = 0$, then $0 \leq (u_t^+)^2(x, 0) = w^\epsilon(x, 0) \leq w^\epsilon(x_0, t_0) = -\epsilon t_0$, which implies $t_0 = 0$.}
	: (1) $t_0 = 0$, (2) $t_0 > 0$ and $(u_t^+)^2(x_0, t_0) > 0$.
	When $t_0 = 0$, $w^\epsilon \leq \sup_{\{t=0\}}(u_t^+)^2$. When we take $\epsilon \rightarrow 0$, we have
	\begin{equation}\label{u_t_upper_case_2}
	e^{2\lambda t}(u_t^+)^2 \leq \sup_{\{t=0\}}(u_t^+)^2.
	\end{equation}  
	When $t_0 > 0$ and $(u_t^+)^2(x_0, t_0) > 0$,
	by (\ref{u_tt}) and ellipticity of $F$, we have near $(x_0, t_0)$
	\begin{equation}\label{w_t_upper}
	\mathcal{L}[u] w^\epsilon = 2 \big(\psi_z - \lambda\big) e^{2\lambda t} u_t^2 + \epsilon + 2 e^{2\lambda t}F^{ij}(u_t)_i (u_t)_j \geq \epsilon,
	\end{equation}
	if we choose $\lambda \leq \inf_{\overline{M_T}} \psi_z$.
	If $x_0 \in M$, then by (\ref{w_t_upper}), $0 \geq \mathcal{L}[u] w^\epsilon \geq \epsilon$, which is a contradiction. 
	If $x_0 \in \partial M$, then by (\ref{w_t_upper}) and strong maximum principle, $w^\epsilon_\nu(x_0, t_0) < 0$, which contradicts to $w^\epsilon_\nu = 2\phi_z e^{2\lambda t} u_t^2 \geq 0$. Hence, (\ref{u_t_upper_case_2}) holds and we obtain the upper bound of $u_t$.
	
	To obtain the lower bound of $u_t$, we just replace $u_t^+$ with $u_t^- := \min \{u_t, 0\}$ in the definition of $w^\epsilon$.
\end{proof}

$C^0$ estimates follow immediately.
\begin{corollary}\label{para_C_0}
	\begin{equation}
	\sup_{\overline{M_T}} |u| \leq \sup_{\overline{M}\times \{t=0\}} |u| + \sup_{\overline{M_T}} |u_t|\, T.
	\end{equation}
\end{corollary}

Next, we give the gradient estimates.

\begin{theorem}\label{grad_estimate}
	Suppose $u\in C^{3, 1}(M_T) \cap C^{2,1}(\overline{M_T})$, then 
	\begin{equation}\label{interior_gradient}
	\sup_{\overline{M_T}} |\nabla u|^2 \leq C,
	\end{equation}
	where $C$ depends on $|u|$, $|u_t|$, $||u_0||_{C^1}$, $||\chi||_{C^1}$, $||\psi||_{C^1}$, $||\phi||_{C^3}$, $||d||_{C^3}$, and the background geometric data, but is independent of $|\psi_t|$.
	Here $d$ is defined in (\ref{dist2bd_function}).
\end{theorem}

\begin{proof}
	The proof is motivated by \cite{Guan18}, \cite{GuanXiang18}.
	
	We extend $\phi(x, z)$ smoothly to $\overline{M} \times \mathbb{R}$.
	Let $$v := u - \phi d, \qquad w := 1 + |\nabla v|^2,$$ 
	and 
	\begin{equation}
	\eta := -v + 1 + \sup_{\overline{M_T}} v + 0.5 c_0(1 - Bd),
	\end{equation}
	where $d \;(\leq \delta_0)$ is defined as in (\ref{dist2bd_function}), $c_0 := 2\big(1 + 2 \sup_{\overline{M_T}}|u| + 2\sup_{\overline{M_T}}|\phi(x, u)| \big)$, and $B$ is a positive constant (independent of $d$) to be determined later. 
	Obviously, $c_0$ does not depend on $d$. 
	If we choose $0<\delta_0<1$ small enough (depending on $B$) such that $1 - Bd \geq 0$, then
	\begin{equation}
	1 \leq \eta \leq  1 + 2 \sup_{\overline{M_T}}|v| +  0.5 c_0 < 1 + 2 \sup_{\overline{M_T}}|u| + 2\sup_{\overline{M_T}}|\phi(x, u)| +  0.5 c_0 = c_0.
	\end{equation}
	For $\epsilon > 0$ to be determined later, we can choose $0<\delta_0<1$ small enough such that 
	\begin{equation}\label{C1_delta_0_constraint}
	0.5 \leq 1 - \phi_u d \leq 2, \qquad \big( |\phi_{uu}| + |\phi_{uuu}| \big) d \leq \epsilon \qquad \text{on}\; \overline{M_T}.
	\end{equation}
	Since $w \leq 2(1 - \phi_u d)^2 |\nabla u|^2 + \tilde{C} \leq 8 |\nabla u|^2 + \tilde{C}$ and $w \geq \frac{1}{2} (1 - \phi_u d)^2 |\nabla u|^2 - \tilde{C} \geq \frac{1}{8} |\nabla u|^2 - \tilde{C}$, then
	\begin{equation}\label{C1_w_gra_u}
	\frac{1}{16}|\nabla u|^2 \leq w \leq 9|\nabla u|^2,
	\end{equation}
	where we assume that $|\nabla u|^2 \geq 16\tilde{C}$, otherwise, (\ref{interior_gradient}) is proved. To prove (\ref{interior_gradient}), it suffices to prove $w \leq C$.
	
	Assume $\log w - \log \eta$ attains a maximum at $(x_0, t_0) \in \overline{M_T}$. Without loss of generality, we assume $t_0 > 0$, otherwise, $w(x_0, 0) \leq C_0$ is determined by the initial data and $w \leq c_0 C_0$.
	
	We first consider the case of $x_0 \in \partial M$. Since $\nabla_{\nu} u = \phi$, $d = 0$ and $\nabla_{\nu}d = 1$ on $\partial M$, then $\nabla_{\nu} v = 0$ on $\partial M$. We choose a local orthonormal frame $\{e_i\}$ near $x_0$ such that $e_n = \nu$ on $\partial M$ and $\nabla v = |\nabla v|e_1$ at $x_0$. 
	Since $g(\nabla_\nu e_1, e_1) = 0$, then at $(x_0, t_0)$,
	\begin{equation}
	\nabla_\nu (\nabla_1 v) = \nabla_1 (\nabla_{\nu} v) - g(\nabla v, \nabla_{e_1} \nu) + g(\nabla v, \nabla_{\nu} e_1) = - |\nabla v| g(e_1, \nabla_{e_1} \nu),
	\end{equation}
	and we obtain the following contradiction
	\begin{equation}
	\begin{aligned}
	0 & \geq \eta \nabla_{\nu} w - w \nabla_{\nu} \eta\\
	& = 2\eta \nabla_1 v \, \nabla_{\nu}(\nabla_1 v) + 0.5 c_0 Bw\\
	& = - 2\eta g(e_1, \nabla_{e_1} \nu) |\nabla v|^2 + 0.5 c_0 Bw\\
	& > \big(0.5 c_0 B - 2\eta g(e_1, \nabla_{e_1} \nu) \big)|\nabla v|^2 \geq 0,
	\end{aligned}
	\end{equation}
	if we choose $B = 1 + \max \{ 4 \sup_{x\in \partial M} \sup_{\tau \in T_x \partial M, |\tau|=1} g(\tau, \nabla_\tau \nu) (x), 0 \}$.

	It remains to consider the case of $x_0 \in M$.
	Since $v = u - \phi d$, we have
	\begin{equation}
	v_t = (1 - \phi_u d) u_t,
	\end{equation}
	\begin{equation}
	v_j = (1 - \phi_u d) u_j -\phi_j d - \phi d_j,
	\end{equation}
	\begin{equation}\label{C1_v_ij}
	\begin{aligned}
	v_{ij} =& (1 - \phi_u d) u_{ij} - \phi_{uu} d\, u_i u_j
	- (\phi_u d_j + \phi_{ju}d) u_i - (\phi_u d_i + \phi_{iu}d) u_j\\
	& - \phi_{ij} d - \phi d_{ij} - \phi_i d_j - \phi_j d_i,
	\end{aligned}
	\end{equation}
	\begin{equation}\label{C1_v_ijk}
	\begin{aligned}
	v_{ijk} =& (1 - \phi_u d) u_{ijk} - (\phi_{uu}d\, u_k + \phi_u d_k + \phi_{ku} d)u_{ij} - \phi_{uuu}d\,u_i u_j u_k\\
	&  - (\phi_{uu}d\, u_j + \phi_u d_j + \phi_{ju}d) u_{ik} - (\phi_{uu}d\, u_i + \phi_u d_i + \phi_{iu}d) u_{jk} + Y_{ijk},
	\end{aligned}
	\end{equation}
	where $|Y_{ijk}| \leq C(1 + |\nabla u|^2) \leq Cw$.

	We choose an orthonormal frame $\{e_i\}$ around $x_0$ such that $\nabla_{e_i} e_j (x_0) = 0$ and $U_{ij} := \chi_{ij} + u_{ij}$ is diagonalized at $(x_0, t_0)$.
	Then $(F^{ij}) := \big(F^{ij}(U)\big)$ is also diagonalized at $(x_0, t_0)$.
	At $(x_0, t_0)$, we have
	\begin{equation}\label{grad101}
	\frac{w_j}{w} - \frac{\eta_j}{\eta} = 0,
	\end{equation}
	\begin{equation}\label{grad102}
	\frac{w_t}{w} - \frac{\eta_t}{\eta} \geq 0,
	\end{equation}
	and
	\begin{equation}\label{grad103}
	F^{ij} \left( \frac{w_{ij}}{w} - \frac{w_i w_j}{w^2} - \frac{\eta_{ij}}{\eta} + \frac{\eta_i \eta_j}{\eta^2} \right) \leq 0.
	\end{equation}
	
	Since $\sum_i F^{ii}U_{ii} \geq 0$ and $\sum_i F^{ii} \geq \gamma$ for some constant $\gamma$ depending on the level set of $F$, by (\ref{C1_w_gra_u}), (\ref{C1_v_ij}) (\ref{grad101}) and (\ref{grad103}), we have
	\begin{equation}\label{F_ij*w_ij_right_ext}
	\begin{aligned}
	F^{ij} w_{ij} 
	&\leq \frac{w}{\eta} F^{ij} \eta_{ij}
	= \frac{w}{\eta} F^{ij} (-v_{ij} - 0.5 c_0 B d_{ij})\\
	&\leq \frac{w}{\eta} \Big( - (1 - \phi_u d) F^{ij} U_{ij} + \phi_{uu}\,d\, F^{ij}u_i u_j  \\
	& \qquad + 2 F^{ij} (\phi_u d_j + \phi_{ju}d) u_i + C \sum_i F^{ii} \Big)\\
	& \leq \epsilon C_1 w^2 \sum_i F^{ii} + C w^{3/2} \sum_i F^{ii},
	\end{aligned}
	\end{equation}
	where $C_1$ does not depend on $d$.

	By (\ref{C1_v_ij}) and Cauchy-Schwarz inequality, 
	\begin{equation}
	\begin{aligned}
	\sum_k F^{ij} v_{ki} v_{kj} & = \sum_k F^{ij}\big((1-\phi_u d) U_{ik} - Z_{ik}\big)\big((1-\phi_u d) U_{jk} - Z_{jk}\big) \\
	& \geq \frac{3}{4} (1-\phi_u d)^2 \sum_k F^{ij}  U_{ik}U_{jk} - 3 \sum_k F^{ij} Z_{ik}Z_{jk} \\
	& \geq \frac{3}{4} (1-\phi_u d)^2 F^{ii}  U_{ii}^2 - 6 (\phi_{uu} d)^2 |\nabla u|^2 F^{ii} u_i^2 - C w \sum_i F^{ii} \\
	& \geq \frac{3}{4} (1-\phi_u d)^2 F^{ii}  U_{ii}^2 - C_1 \epsilon^2 w^2 \sum_i F^{ii} - Cw\sum_i F^{ii},
	\end{aligned}
	\end{equation}
	where $Z_{ij} := \phi_{uu} d\, u_i u_j +  (\phi_u d_j + \phi_{ju}d) u_i + (\phi_u d_i + \phi_{iu}d) u_j + \phi_{ij} d + \phi d_{ij} + \phi_i d_j + \phi_j d_i + (1 - \phi_u d)\chi_{ij}$ and the constant $C_1$ does not depend on $d$.
	Differentiating (\ref{main_eq}),
	\begin{equation}\label{1st_diff}
	u_{tk} = F^{ij} (\chi_{ij,k} + u_{ijk}) - (\psi)_k.
	\end{equation}
	Since $w_t = 2 \sum_k v_k v_{kt}$, then by (\ref{1st_diff}),
	\begin{equation}\label{C_1_w_t}
	\begin{aligned}
	&\quad (1 - \phi_u d) \sum_k F^{ij} u_{ijk} v_k\\
	&= 0.5 w_t - (1 - \phi_u d) \sum_k v_k \left[ F^{ij}\chi_{ij,k} - (\psi)_k \right]\\
	& \quad + (\phi_{uu}d) u_t \sum_k u_k v_k + u_t d \sum_k \phi_{ku} v_k + \phi_u u_t \sum_k d_k v_k.
	\end{aligned}
	\end{equation}
	By (\ref{ijk-ikj}), (\ref{C1_delta_0_constraint}), (\ref{C1_v_ijk}), (\ref{grad102}), (\ref{C_1_w_t}), $\sum_i F^{ii} \geq \gamma$, and  Cauchy-Schwarz inequality,
	\begin{equation}\label{F_ij*w_ij_left_ext_2}
	\begin{aligned}
	\sum_k F^{ij} v_k v_{kij} & =  \sum_k F^{ij} v_k \big( v_{ijk} - R^m_{ijk} v_m \big) \\
	& \geq (1 - \phi_u d) \sum_k F^{ij}u_{ijk} v_k  - \phi_{uuu}d\,\sum_k F^{ij} u_i u_j u_k v_k
	\\
	& \,\,\,\,\,\,\, - \sum_k (\phi_{uu}d\, u_k + \phi_u d_k + \phi_{ku} d) v_k\, F^{ij}u_{ij} \\
	& \,\,\,\,\,\,\, - 2 \sum_k F^{ij} (\phi_{uu}d\, u_j + \phi_u d_j + \phi_{ju}d) u_{ik} v_k - Cw^{3/2}\sum_i F^{ii} \\
	& \geq  - C_1 |\phi_{uuu} d| w^2 \sum_i F^{ii} - C_1 w |\phi_{uu} d|\, \sum_i F^{ii}|U_{ii}|\\
	& \,\,\,\,\,\,\, - C\sqrt{w}\sum_i F^{ii}|U_{ii}| - Cw^{3/2}\sum_i F^{ii} + 0.5 w_t \\
	& \geq   - 0.5\epsilon \sum_i F^{ii}U_{ii}^2 - C_1 \Bigg( |\phi_{uuu} d| + \frac{|\phi_{uu} d|^2}{\epsilon} \Bigg) w^2 \sum_i F^{ii} \\
	& \quad - \frac{C}{\epsilon} w^{3/2}\sum_i F^{ii} -\frac{ (1 - \phi_u d) u_t w}{2\eta} \\
	& \geq - 0.5\epsilon F^{ii}U_{ii}^2 - \epsilon C_1 w^2 \sum_i F^{ii}  - \frac{C}{\epsilon} w^{3/2}\sum_i F^{ii},
	\end{aligned}
	\end{equation}
	where the constant $C_1$ does not depend on $d$.
	Hence,
	\begin{equation}\label{F_ij*w_ij_left_ext}
	\begin{aligned}
	F^{ij} w_{ij} & = 2 \sum_k F^{ij} v_{ki} v_{kj} + 2 \sum_k F^{ij} v_k v_{kij}\\
	& \geq \frac{3}{2} (1-\phi_u d)^2 F^{ii}  U_{ii}^2 - \epsilon F^{ii}U_{ii}^2 - \epsilon C_1 w^2 \sum_i F^{ii}  - Cw^{3/2}\sum_i F^{ii}\\
	& \geq (1-\phi_u d)^2 F^{ii}  U_{ii}^2 - \epsilon C_1 w^2 \sum_i F^{ii}  - Cw^{3/2}\sum_i F^{ii},
	\end{aligned}
	\end{equation}
	where we assume $\epsilon \leq 1/8 \leq (1-\phi_u d)^2/2$, and $C_1$ is independent of $d$.

	Let $I:= \{ i : \sqrt{n}|v_i| \geq |\nabla v| \}$. Obviously, $I\neq \emptyset$. For $i\in I$, since $w_i = 2 \sum_k v_k v_{ki}$, by (\ref{grad101}),
	\begin{equation}
	\begin{aligned}
	(1 - \phi_u d)U_{ii} & \leq \frac{w_i}{2v_i} + C_1 |\phi_{uu} d| \frac{w^{3/2}}{|v_i|} + C \frac{w}{|v_i|} = \frac{w \eta_i}{2\eta v_i} + C_1 |\phi_{uu} d| \frac{w^{3/2}}{|v_i|} + C \frac{w}{|v_i|} \\
	& \leq -\frac{ w}{2 c_0} + C_1 \epsilon w + C_2 \frac{w}{|\nabla v|} \leq -\frac{ w}{4 c_0},
	\end{aligned}
	\end{equation}
	where we assume that $\epsilon \leq \frac{1}{8 c_0 C_1}$ and $|\nabla v| \geq 8 c_0 C_2$, and the constant $C_1$ does not depend on $d$. Let $U_{mm} = \min_i \{U_{ii} \}$, then $(1 - \phi_u d) U_{mm} \leq (1 - \phi_u d) U_{ii} \leq -\frac{ w}{4 c_0}$, and $F^{mm} \geq \frac{1}{n} \sum_i F^{ii}$, and hence,
	\begin{equation}\label{f_lambda^2_m}
	(1 - \phi_u d)^2 F^{mm}U_{mm}^2 \geq \frac{w^2}{16nc_0^2}\sum_i F^{ii}.
	\end{equation}
	
	By (\ref{F_ij*w_ij_right_ext}), (\ref{F_ij*w_ij_left_ext}), (\ref{f_lambda^2_m}) and $\sum_i F^{ii} \geq \gamma'$, we have
	\begin{equation}
	\begin{aligned}
	0 & \geq (1-\phi_u d)^2 F^{ii}  U_{ii}^2 - \epsilon C_1 w^2 \sum_i F^{ii}  - \frac{C}{\epsilon} w^{3/2}\sum_i F^{ii}\\
	& \geq \frac{w^2}{16nc_0^2}\sum_i F^{ii} - \epsilon C_1 w^2 \sum_i F^{ii}  - \frac{C}{\epsilon} w^{3/2}\sum_i F^{ii}\\
	& \geq \frac{w^2}{32nc_0^2}\sum_i F^{ii} - C w^{3/2}\sum_i F^{ii},
	\end{aligned}
	\end{equation}
	where we assume $\epsilon \leq \frac{1}{32nc_0^2 C_1}$, and the constant $C_1$ does not depend on $d$.	
	Hence, $w(x_0, t_0) \leq C$. For any $(x, t) \in \overline{M_T}$, 
	\begin{equation}
	w(x, t) \leq \frac{\eta(x, t)}{\eta(x_0, t_0)} w(x_0, t_0) \leq C.
	\end{equation}
\end{proof}

Next, we derive the second-order estimates. Due to technical difficulties, we assume $\overline{M}$ satisfy certain curvatures conditions. 

The following proposition about (parabolic) $\mathcal{C}$-subsolutions is crucial for the second-order estimates.
\begin{proposition}\label{para_sub_key_1}
	Suppose $\underline{u}\in C^2$ is a (parabolic) $\mathcal{C}$-subsolution of the parabolic equation $F(\chi_{ij} + u_{ij}) - u_t = \psi$ on compact manifolds $\overline{M}$.
	Then there exists $\theta > 0$ depending only on $F$, $\chi$, $\underline{u}$, $\inf_{\overline{M}} \psi$ and $\sup_{\overline{M}} \psi$ such that one of the following is true
	\begin{enumerate}
		\item 
		\begin{equation}
		\sum_{i,j} F^{ij} \big( \underline{u}_{ij} - u_{ij} \big) - (\underline{u}_t - u_t) \geq \theta \sum_{i} F^{ii} + \theta,
		\end{equation}
		\item 
		\begin{equation}
		\lambda_{min}(F^{ij}) \geq \theta \sum_{i} F^{ii} + \theta,
		\end{equation}
	\end{enumerate}
	where $F^{ij} := F^{ij}(\chi_{ij} + u_{ij})$, and $\lambda_{min}(F^{ij})$ is the smallest eigenvalue of $(F^{ij})$.
\end{proposition}
\begin{proof}
	The proof for $F(A) = f\big(\lambda(A)\big)$ being a function of eigenvalues can be found in Lemma 3 in \cite{PhongTo17}. The proof for a more general function $F$ which is not necessary to be a function of eigenvalues can be found in Proposition 2.1.17 in \cite{GuoThesis19}.
\end{proof}

\begin{theorem}\label{C2_interior_estimate}
	Suppose $(\overline{M}, g)$ has non-negative sectional curvatures, and the principal curvatures of $\partial M$ is bounded below by a positive constant $\underline{\kappa}$ with $2\underline{\kappa} + \inf_{\partial M} \phi_u > 0$. 
	Suppose either $\psi_u \equiv 0$ or $F(A) = \log \sigma_k \big(\lambda(A)\big)$.
	Suppose $u\in C^{4, 1}(M_T) \cap C^{3, 1}\big(\overline{M}\times (0, T]\big) \cap C^{2, 1}(\overline{M_T})$, and $\underline{u} \in C^{2, 1}\big(\overline{M_T}\big)$ is a (parabolic) $\mathcal{C}$-subsolution of $F(\chi_{ij} + u_{ij}) = u_t + \psi$. Then
	\begin{equation}\label{C2_interior}
	\sup_{\overline{M_T}} |\nabla^2 u| \leq C \left(1 + \sup_{\overline{M_T}}|\phi_{uuu}|\, |\nabla u|\right) \left( 1 + \sup_{\overline{M_T}} |\nabla u|^2 \right) + C \sup_{\partial M \times (0, T]} |u_{\nu\nu}|,
	\end{equation}
	where the constant $C$ depends on $|u|$, $|u_t|$,  $||\phi||_{C^3}$, $||\psi||_{C^2}$, $||u_0||_{C^{2}}$, $||\underline{u}||_{C^{2, 1}}$, $||\chi||_{C^2}$, $||d||_{C^4}$, $\big\{(\chi_{ij} + \underline{u}_{ij})\big\}$, $2\underline{\kappa} + \inf_{\partial M} \phi_u$ and the back ground geometric data, but it is independent of $|\psi_t|$. 
	Here $d$ is defined in (\ref{dist2bd_function}).
\end{theorem}

We extend $\nu$ in $\overline{M}$ by $\nu = \nabla d$, where $d$ is defined in (\ref{dist2bd_function}). Motivated by \cite{LionsTrudUrbas86}, \cite{MaQiu19} and \cite{GuanXiang18}, for $x \in \overline{M}$, and $\xi, \sigma \in T_x M$, we consider the following symmetric $(0, 2)$-tensors
\begin{equation}\label{tilde(W)}
\begin{aligned}
\tilde{W}(\xi, \sigma) :=& - g(\xi, \nu) \Big( \nabla_{\sigma'}(\phi) - \nabla_{\nabla_{\sigma'}\nu} u  + \chi_{\sigma'\nu}\Big) \\
& \quad - g(\sigma, \nu) \Big( \nabla_{\xi'}(\phi) - \nabla_{\nabla_{\xi'}\nu} u  + \chi_{\xi'\nu}\Big),
\end{aligned}
\end{equation}
and
\begin{equation}\label{W_func}
W(\xi, \sigma) := U_{\xi\sigma} + \tilde{W}(\xi, \sigma),
\end{equation}
where $U_{ij} := \chi_{ij} + u_{ij}$ and $\xi' := \xi - g(\xi, \nu)\nu$.

First, we consider the interior second-order estimates.

Suppose $\sup_{\xi \in T_x \overline{M},\, |\xi|=1} \left[W(\xi , \xi)(x, t) + \eta\right] $ is attained at $(x_0, t_0) \in M_T$ with $\xi = \tau \in T_{x_0} \overline{M}$ and $|\tau| = 1$, where $\eta = \eta (|\nabla u|, u)$ is to be determined later.
We choose an orthonormal frame $\{e_i\}$ around $x_0$ such that $\nabla_{e_i} e_j (x_0) = 0$ and the matrix $\left( U_{ij} \right)(x_0, t_0)$ is diagonalized with $U_{11}(x_0, t_0)$ as the largest eigenvalue. Note that $(F^{ij}) := \big(F^{ij}(U)\big)$ is also diagonalized at $(x_0, t_0)$. We extend $\tau$ to a neighborhood of $x_0$ such that $\nabla_{e_i} \tau (x_0) = 0$ for all $1 \leq i \leq n$. Then at $(x_0, t_0) \in M_T$,
\begin{equation}
U_{\tau\tau, i} + \tilde{W}_{\tau\tau, i} + \eta_i = 0,
\end{equation}
\begin{equation}\label{C2_CrPt_2nd_order}
F^{ij} \left( u_{\tau\tau ij} + \chi_{\tau\tau, ij} + \tilde{W}_{\tau\tau, ij} + \eta_{ij} \right) - u_{\tau\tau t} - \tilde{W}_{\tau\tau, t} - \eta_t \leq 0.
\end{equation}
Differentiating (\ref{main_eq}) once and twice, we have
\begin{equation}\label{1st_diff_C2}
u_{t\tau} = F^{ij} (\chi_{ij,\tau} + u_{ij\tau}) - (\psi)_\tau
\end{equation}
\begin{equation}\label{2nd_diff_C2}
u_{t\tau\tau} = F^{ij} u_{ij\tau\tau} + F^{ij} \chi_{ij,\tau\tau} + F^{ij,kl} U_{ij,\tau} U_{kl,\tau} - (\psi)_{\tau\tau}.
\end{equation}
By (\ref{ijk-ikj}) and (\ref{1st_diff_C2}),
\begin{equation}\label{C2_ineq_7}
\begin{aligned}
F^{ij} u_{kij} - u_{kt} &= F^{ij} u_{ijk} - F^{ij} u_m R^m_{ijk} - u_{tk}\\
&= (\psi)_k - F^{ij} \chi_{ij,k} - F^{ij} u_m R^m_{ijk}.
\end{aligned}
\end{equation}
By (\ref{ijkl-klij}) and (\ref{2nd_diff_C2}),
\begin{equation}\label{C2_ineq_1}
\begin{aligned}
& \qquad u_{\tau\tau t} - F^{ij} u_{\tau\tau ij} - F^{ij} \chi_{\tau\tau, ij}\\
&= -(\psi)_{\tau\tau} + F^{ij} \chi_{ij,\tau\tau} - F^{ij} \chi_{\tau\tau, ij} + F^{ij,kl} U_{ij,\tau} U_{kl,\tau} \\
& \quad -2F^{ij} u_{im} R^m_{\tau\tau j} - 2F^{ij} u_{m\tau} R^m_{i\tau j} - F^{ij} u_m \big( R^m_{\tau\tau i,j} + R^m_{i\tau j,\tau} \big)
\end{aligned}
\end{equation}
Since $W_{\tau\tau} (x_0, t_0)$ is the largest eigenvalue of $W(\cdot, \cdot)(x_0, t_0)$, then there exists an orthonormal frame $\{ E_\alpha \}_{\alpha=1}^{n}$ at point $x_0$ such that $E_1 = \tau$ and the $n\times n$ matrix $(W_{\alpha\beta})(x_0, t_0)$ is diagonalized. Let $e_i = a^\alpha_i E_\alpha$, and $F^{\alpha\beta} := a^\alpha_i F^{ij} a^{\beta}_j$, then the matrix $(F^{\alpha\beta})$ is positive definite. Since sectional curvatures of $\overline{M}$ are non-negative, then the $n\times n$ matrix $(R_{\tau\alpha\tau\beta})$ ($\tau$ is fixed) is symmetric and positive definite. Hence,
\begin{equation}\label{C2_ineq_2}
\begin{aligned}
- F^{ij} u_{m\tau} R^m_{i\tau j} &= - F^{ij} W _{m\tau} R^m_{i\tau j} + F^{ij} (\chi_{m\tau} + \tilde{W}_{m\tau}) R^m_{i\tau j} \\
&= - F^{\alpha\beta} W_{\tau\tau} R_{\tau\alpha\tau\beta} + F^{ij} (\chi_{m\tau} + \tilde{W}_{m\tau}) R^m_{i\tau j}\\
& \leq  C (1 + |\nabla u|^2) \sum_i F^{ii},
\end{aligned}
\end{equation}
where we assume $W_{\tau\tau} \geq 0$.
By Cauchy-Schwarz inequality,
\begin{equation}\label{C2_ineq_3}
-2F^{ij} u_{im} R^m_{\tau\tau j} \leq \frac{1}{4} \sum_i F^{ii} U_{ii}^2 + C \sum_{i} F^{ii}.
\end{equation}
By concavity of $F$, 
\begin{equation}\label{C2_ineq_4}
F^{ij,kl} U_{ij,\tau} U_{kl,\tau}  \leq 0.
\end{equation}
Combining (\ref{C2_ineq_1}), (\ref{C2_ineq_2}), (\ref{C2_ineq_3}), (\ref{C2_ineq_4}), we have
\begin{equation}\label{C2_ineq_5}
u_{\tau\tau t} - F^{ij} u_{\tau\tau ij} - F^{ij} \chi_{\tau\tau, ij} \leq \frac{1}{4} \sum_i F^{ii} U_{ii}^2 + CK \sum_{i} F^{ii} - (\psi)_{\tau\tau},
\end{equation}
where $K := \max_{\overline{M_T}} (1 + |\nabla u|^2).$
By (\ref{tilde(W)}), (\ref{C2_ineq_7}) and Cauchy-Schwarz inequality, 
\begin{equation}\label{C2_ineq_6}
\begin{aligned}
& \qquad \tilde{W}_{\tau\tau, t} - F^{ij}\tilde{W}_{\tau\tau, ij} \\
& \leq 2 g_{\tau\nu} \Big(\nabla_{\nabla_{\tau'}\nu} u_t - F^{ij} \nabla_i \nabla_j \nabla_{\nabla_{\tau'}\nu} u\Big)
- 2 g_{\tau\nu} \phi_u (u_{\tau' t} - F^{ij} u_{\tau' ij})\\
& \quad + \sum_i F^{ii} Z_i U_{ii} + C K\Big(1 + |\phi_{uuu}| K^{1/2}\Big) \sum_i F^{ii} + CK \\
& \leq  \frac{1}{4} \sum_i F^{ii} U_{ii}^2 + C K\Big(1 + |\phi_{uuu}| K^{1/2}\Big) \sum_i F^{ii} + CK,
\end{aligned}
\end{equation}
where $|Z_i| \leq CK^{1/2}$.
Hence, by (\ref{C2_CrPt_2nd_order}), (\ref{C2_ineq_5}), (\ref{C2_ineq_6}), we have
\begin{equation}\label{C2_ineq_right}
F^{ij} \eta_{ij} - \eta_t \leq \frac{1}{2} \sum_i F^{ii} U_{ii}^2 + C K \Big(1 + |\phi_{uuu}| K^{1/2}\Big) \sum_i F^{ii} + CK + |\psi_u U_{\tau\tau}|.
\end{equation}

Let 
\begin{equation}\label{eta_C2}
\eta := \zeta (|\nabla u|^2) + A (\underline{u} - u),
\end{equation}
where the function $\zeta$ and the constant $A$ are to be determined later. Then
\begin{equation}
\eta_i = 2 \zeta' \sum_k u_k u_{ki} + A (\underline{u} - u)_i,
\end{equation}
\begin{equation}\label{eta_ij_C2}
\eta_{ij} = 4 \zeta'' \sum_{k, l} u_k u_{ki} u_l u_{lj} + 2 \zeta' \sum_k u_{kj} u_{ki} + 2\zeta' \sum_k u_k u_{kij}  + A (\underline{u} - u)_{ij},
\end{equation}
\begin{equation}\label{eta_t_C2}
\eta_t = 2\zeta' \sum_k u_k u_{kt} + A (\underline{u} - u)_t.
\end{equation}
Let $\zeta (z) := Bz$, then $\zeta' = B,\zeta'' =0$, where $B$ is a positive constant to be determined later. Then by (\ref{eta_ij_C2}), (\ref{eta_t_C2})  and (\ref{C2_ineq_7}) and Cauchy-Schwarz inequality, we have
\begin{equation}\label{C2_ineq_left}
\begin{aligned}
& \qquad F^{ij} \eta _{ij} - \eta_t\\
&\geq 2 B \sum_k F^{ij} u_{kj} u_{ki} + A F^{ij} (\underline{u} -u)_{ij} - A(\underline{u} - u)_t - CK(1 + \sum_i F^{ii}) \\
& \geq B \sum_i F^{ii} U_{ii}^2 + A F^{ij} (\underline{u} -u)_{ij} - A(\underline{u} - u)_t - CK(1 + \sum_i F^{ii}).
\end{aligned}
\end{equation}
If $F(A) = \log \sigma_k \big(\lambda(A)\big)$, as in \cite{MaQiu19}, by (\ref{sigma_k_prop_5}), we have
\begin{equation}
F^{11}U_{11}^2 \geq \frac{k}{n} U_{11} \geq \frac{k}{n^2} |U_{\tau\tau}|.
\end{equation}
Hence, we choose $B := 1 + \frac{n^2}{k} \sup |\psi_u|$, which also works for the case of $\psi_u \equiv 0$.
Then by (\ref{C2_ineq_right}) and (\ref{C2_ineq_left}), we have
\begin{equation}\label{C2_ineq_8}
\begin{aligned}
0 &\geq \frac{1}{2} \sum_i F^{ii} U_{ii}^2 + A F^{ij} (\underline{u} -u)_{ij} - A(\underline{u} - u)_t \\
&\quad -  C_1 K \Big(1 + |\phi_{uuu}| K^{1/2}\Big) \sum_i F^{ii} - C_1 K.
\end{aligned}
\end{equation}
By Proposition \ref{para_sub_key_1}, there exists a constant $\theta$ such that one of the following holds:
\begin{align}
&(1) \qquad F^{ij} (\underline{u} - u)_{ij} - (\underline{u} - u)_t \geq \theta \sum_i F^{ii} + \theta,
\label{C2_ineq_9.1}\\
&(2) \qquad F^{kk} \geq \theta \sum_i F^{ii} + \theta, \qquad \forall\, 1 \leq k \leq n. \label{C2_ineq_9.2}
\end{align}
Now we choose $A = \frac{C_1 K \big(1 + |\phi_{uuu}| K^{1/2}\big)}{\theta}$. When (\ref{C2_ineq_9.1}) holds, by (\ref{C2_ineq_8}), we have the following contradiction
\begin{equation}
0 \geq \frac{1}{2} \sum_i F^{ii} U_{ii}^2 + \big[A\theta -  C_1 K \big(1 + |\phi_{uuu}| K^{1/2}\big)\big] \sum_i F^{ii} + A\theta - C_1 K  > 0.
\end{equation}
When (\ref{C2_ineq_9.2}) holds, by (\ref{C2_ineq_8}), we have
\begin{equation}
0 \geq \left[\frac{\theta}{2} U_{11}^2 -  C_1 K \big(1 + |\phi_{uuu}| K^{1/2}\big) \right] \sum_i F^{ii} + \frac{\theta}{2} U_{11}^2 -  C_1 K,
\end{equation}
which implies
\begin{equation}
U_{11} \leq C K^{3/4},
\end{equation}
and hence
\begin{equation}\label{C2_ineq_10}
W_{\tau\tau} \leq U_{\tau\tau} + CK \leq U_{11} + CK \leq C K.
\end{equation}

Next, we consider the second-order estimates on boundary.
On $\partial M$, we have the following lemma.
\begin{lemma}\label{W_bd_lemma}
	For any $x\in \partial M$ and $\xi \in T_x\overline{M}$ with $|\xi| = 1$, we have
	\begin{equation}
	W(\xi, \xi) \leq \max \Big\{ \sup_{\sigma \in T_x \partial M, |\sigma|=1} W(\sigma, \sigma),\; W(\nu, \nu) \Big\}.
	\end{equation}
\end{lemma}
\begin{proof}
	Suppose $\xi \in T_x\overline{M}$ satisfies $|\xi| = 1$ and $W(\xi, \xi) = \sup_{\sigma \in T_x \overline{M}, |\tau|=1} W(\tau, \tau)$. It suffices to consider the case of $0< |\xi'| <1$.
	Since $|\xi|^2 = |\xi'|^2 + \big(g(\xi, \nu)\big)^2 |\nu|^2$, then
	\begin{equation}\label{Lemma_W_1}
	\begin{aligned}
	u_{\xi\xi} &= u_{\xi'\xi'} + 2 g(\xi, \nu) u_{\xi'\nu} + \big(g(\xi, \nu)\big)^2 u_{\nu\nu}\\
	&= u_{\xi'\xi'} + 2 g(\xi, \nu) \left( \nabla_{\xi'}(u_\nu) - \nabla_{\nabla_{\xi'}\nu}u \right) + (1 - |\xi'|^2) u_{\nu\nu}, 
	\end{aligned}
	\end{equation}
	and
	\begin{equation}\label{Lemma_W_2}
	\chi_{\xi\xi} = \chi_{\xi'\xi'} + 2 g(\xi, \nu) \chi_{\xi'\nu} + (1 - |\xi'|^2) \chi_{\nu\nu}.
	\end{equation}
	Since $u_\nu = \phi$ on $\partial M$, then by (\ref{tilde(W)}), (\ref{W_func}), (\ref{Lemma_W_1}), (\ref{Lemma_W_2}),
	\begin{equation}
	\begin{aligned}
	W(\xi, \xi) &= |\xi'|^2 W\Big(\frac{\xi'}{|\xi'|}, \frac{\xi'}{|\xi'|}\Big) + (1 - |\xi'|^2) W(\nu, \nu)\\
	&\leq |\xi'|^2 W(\xi, \xi) + (1 - |\xi'|^2) W(\nu, \nu),
	\end{aligned}
	\end{equation}
	and hence,
	\begin{equation}
	W(\xi, \xi) \leq W(\nu, \nu).
	\end{equation}
\end{proof}

Suppose $\sup_{\xi \in T_x \overline{M},\, |\xi|=1} \left[W(\xi , \xi)(x, t) + \eta\right] $ is attained at $(x_0, t_0) \in \partial M \times (0, T]$ with $\xi = \tau \in T_{x_0} \overline{M}$ and $|\tau| = 1$, where $\eta$ is defined in (\ref{eta_C2}). By Lemma \ref{W_bd_lemma}, either $\tau \in T_{x_0}\partial M$ or $\tau = \nu$. 

Case 1: Suppose $\tau \in T_{x_0}\partial M$. Since for any $\xi \in T_{x_0} \partial M$, $U_{\xi\xi} = W(\xi, \xi) \leq W(\tau, \tau) = U_{\tau\tau}$, we can choose an orthonormal frame $\{e_i\}$ around $x_0$ such that $e_n = \nu$ on $\partial M$, and $e_1 = \tau$ at $x_0$, and the $(n-1)\times (n-1)$ matrix $\left( U_{ij} \right)_{i,j < n}(x_0, t_0)$ is diagonalized. 

Fix $t = t_0$. Since $W(e_1, e_1) + \eta = U_{11} + \eta$ attains a local maximum at $x_0  \in \partial M$, then at $(x_0, t_0)$ we have
\begin{equation}\label{U_11 e^eta_bd}
0 \geq \nabla_n U_{11} + \eta_n.
\end{equation}
For any $k<n$, since $e_k \in T_{x_0}\partial M$ and $g(\nabla_k e_n, e_n) = 0$, then 
\begin{equation}\label{u_kn}
u_{nk} = \nabla_k(\nabla_n u) - \nabla_{\nabla_k e_n}u = \phi_k + \phi_u u_k -\sum_{l<n} \Gamma_{kn}^l u_l,
\end{equation}
and by Lemma \ref{lemma_u_ijk},
\begin{equation}\label{u_11n}
\begin{aligned}
u_{n11} &= \nabla_{1}\nabla_1(u_n) - 2\sum_i \Gamma_{1n}^i u_{1i} - \nabla_{\nabla_{1}\nabla_1 e_n}u\\
&= \phi_u U_{11} - \phi_u \chi_{11} + \phi_{11} + 2\phi_{1u} u_1 + \phi_{uu} u_1^2\\
& \quad - 2\Gamma_{1n}^1 U_{11} + 2\sum_{i<n}\Gamma_{1n}^i \chi_{1i} - 2\Gamma_{1n}^n u_{1n} - \nabla_{\nabla_{1}\nabla_1 e_n}u\\
&\geq (\phi_u - 2\Gamma_{1n}^1) U_{11} - C(1 + |\nabla u|^2).
\end{aligned}
\end{equation}
Since $\Gamma_{n1}^1 = g(\nabla_n e_1, e_1) = 0$ under the orthonormal frame, then by (\ref{ijk-ikj}), (\ref{u_kn}) and (\ref{u_11n}),
\begin{equation}\label{nabla_n U_11}
\begin{aligned}
\nabla_n U_{11} &= \nabla_n u_{11} + \nabla_n \chi_{11}\\
&= u_{11n} + 2 \sum_{k=2}^{n-1}\Gamma_{n1}^{k}u_{1k} + 2\Gamma_{n1}^n u_{1n} + \nabla_n \chi_{11}\\
&= u_{n11} + R_{11n}^m u_m - 2 \sum_{k=2}^{n-1}\Gamma_{n1}^{k}\chi_{1k} + 2\Gamma_{n1}^n u_{1n} + \nabla_n \chi_{11}\\
&\geq (\phi_u - 2\Gamma_{1n}^1) U_{11} - C(1 + |\nabla u|^2).
\end{aligned}
\end{equation}
By (\ref{eta_C2}) and (\ref{u_kn}), 
\begin{equation}\label{eta_n_C2}
\eta_n = 2 \sum_{k}\zeta' u_k u_{kn} + A (\underline{u}_n - u_n) \geq - C (1 + |\phi_{uuu}| K^{1/2})K - 2B |\phi|\, |u_{\nu\nu}|.
\end{equation}
By (\ref{U_11 e^eta_bd}), (\ref{nabla_n U_11}) and (\ref{eta_n_C2}),
\begin{equation}
(\phi_u - 2\Gamma_{1n}^1)  U_{11} \leq C (1 + |\phi_{uuu}| K^{1/2}) K + 2B |\phi|\,|u_{\nu\nu}|.
\end{equation}
Since $\phi_u -2 \Gamma_{1n}^1 = \phi_u + 2 b_{11} \geq 2\underline{\kappa} + \inf_{\partial M} \phi_u > 0$, where $b_{ij}$ is the second fundamental form of $\partial M$, then
\begin{equation}\label{C2_ineq_11}
W_{\tau\tau} = U_{11} \leq C(1 + |\phi_{uuu}| K^{1/2}) K + \frac{2B|\phi|}{2\underline{\kappa} + \inf_{\partial M} \phi_u}|u_{\nu\nu}|.
\end{equation}

Case 2: Suppose $\tau = \nu$. Then obviously,
\begin{equation}\label{C2_ineq_12}
W_{\tau\tau} \leq CK + |u_{\nu\nu}|.
\end{equation}

Suppose $\sup_{\xi \in T_x \overline{M},\, |\xi|=1} \left[W(\xi , \xi)(x, t) + \eta\right] $ is attained at $(x_0, t_0) = (x_0, 0)$ and $\tau \in T_{x_0}$ with $|\tau| = 1$, then $W_{\tau\tau} \leq C$.

In conclusion, for any $(x, t) \in \overline{M_T}$ and $\xi \in T_x \overline{M}$ with $|\xi| = 1$, by (\ref{C2_ineq_10}), (\ref{C2_ineq_11}), (\ref{C2_ineq_12}), we have
\begin{equation}\label{C2_ineq_13}
\begin{aligned}
U_{\xi\xi}(x, t) &\leq W(\xi, \xi)(x, t) + CK \leq W(\tau, \tau)(x_0, t_0) + CK \\
&\leq C(1 + |\phi_{uuu}| K^{1/2})K + \gamma \sup_{\partial M \times [0, T]} |u_{\nu\nu}|,
\end{aligned}
\end{equation}
where $\gamma := \max \{ 1,\frac{2B \sup_{\partial M} |\phi|}{2\underline{\kappa} + \inf_{\partial M} \phi_u}  \}$.
Since $\Gamma \subset \Gamma_1$, then (\ref{C2_interior}) holds.

\section{Estimates of $u_{\nu\nu}$ for $k$-Hessian equations}\label{estimates.sec_3}

We will derive the second-order normal-normal estimates for the IBV problem (\ref{main_eq}) when $F = \log \sigma_k$ for $k\geq 2$. For $k=1$, it is a uniformly parabolic equation. 

\begin{theorem}\label{sigma_k_u_nn_estimates}
	Let $F(A) = \log \sigma_k \big(\lambda(A)\big)$ on $\Gamma_k$. 
	Suppose the principal curvatures of $\partial M$ is bounded below by a positive constant $\underline{\kappa}$ with $\underline{\kappa} > 0$.
	Suppose $u\in C^{4,1}(M_T) \cap C^{3,1}\big(\overline{M_T}\big)$ is an admissible solution of the IBV problem (\ref{main_eq}) satisfying (\ref{C2_interior}), then
	\begin{equation}\label{C2_nn_ineq_0}
	M := \sup_{\partial M \times [0, T]} |u_{\nu\nu}| \leq C,
	\end{equation}
	where the constant $C$ depends on $||u||_{C^{1,1}}$, $||\phi||_{C^3}$, $||\psi||_{C^2}$, $||u_0||_{C^{3, 1}}$, $||\underline{u}||_{C^{2,1}}$, $||\chi||_{C^2}$, $||d||_{C^4}$, $\big\{(\chi_{ij} + \underline{u}_{ij})\big\}$, $2\underline{\kappa} + \inf_{\partial M} \phi_u$ and the back ground geometric data, but it is independent of $|\psi_t|$.
	Here $d$ is defined in (\ref{dist2bd_function}).
\end{theorem}

Theorem \ref{sigma_k_u_nn_estimates} is a direct conclusion of Proposition \ref{C2_nn_sup_u_nn} and Proposition \ref{C2_nn_inf_u_nn} in the following.
The proof is a motivated by Ma-Qiu \cite{MaQiu19}.

Let $$\rho (x) := d - d^2,$$ where $d$ is defined as in (\ref{dist2bd_function}). Let $$M_{\delta} := \{ x \in M : 0< d(x) < \delta \} \qquad \text{and}\qquad M_{\delta, T} := M_{\delta} \times (0, T]$$ with $0 < \delta < \frac{\delta_0}{2}$ to be determined later. 

Given any $x \in \partial M$, we choose an orthonormal frame $\{e_i\}$ in a neighborhood of $x$ such that $e_1, \cdots, e_{n-1} \in T_{x}\partial M$ and $e_n = \nu$ at $x$. Since $\nabla d = \nu$ at $x$, $|\nabla d| \equiv 1$ near $x$, and $\nabla_i \nabla_j d = \nabla_i (\nabla_j d) - \nabla_{\nabla_i e_j}d$, then at $x$ we have
\begin{equation}
\big(\rho_{ij}\big)_{n\times n} = \begin{pmatrix} 
-\big( h_{ij}\big)_{(n-1)\times (n-1)} & 0 \\
0 & -2 
\end{pmatrix},
\end{equation}
and
\begin{equation}
-\max\{\overline{\kappa}, 2\} g \leq \nabla^2 \rho \leq - \min\{ \underline{\kappa}, 2 \} g,
\end{equation}
where $(h_{ij})$ is the second fundamental form, and $\overline{\kappa}$ and $\underline{\kappa}$ are upper and lower bounds of principal curvatures of $\partial M$. 
If we choose $0 < \delta_0 \leq \frac{1}{2}$ small enough such that for any $x\in M_\delta$, we have
\begin{equation}
-\gamma_0 g \leq \nabla^2 \rho \leq - \gamma_1 g,
\end{equation}
and
\begin{equation}
\frac{1}{2} \leq |\nabla \rho| = (1 - 2d) |\nabla d| \leq 1,
\end{equation}
\begin{equation}
\frac{1}{2}d \leq \rho \leq d,
\end{equation}
where
$$\gamma_0 := 2\max\{\overline{\kappa}, 2\}, \qquad \gamma_1 := \frac{1}{2}\min\{ \underline{\kappa}, 2 \}.$$

By (\ref{C2_interior}), there exists $C_0 > 0$ such that for any $(x, t) \in \overline{M_T}$ and $\xi \in T_x \overline{M}$ with $|\xi| = 1$, we have
\begin{equation}\label{C2_nn_0}
|U_{\xi\xi}| \leq C_0 (1 + M).
\end{equation}

\begin{proposition}\label{C2_nn_sup_u_nn}
	There exist large $A, \beta$ and small $\delta$ such that
	\begin{equation}
	W := - (1 + \beta \rho) \left( \nabla_{\nabla \rho} u - \phi - q \right) + (A + \frac{M}{2}) \rho \geq 0 \qquad \text{in} \; \overline{M_{\delta, T}},
	\end{equation}
	where $q(x) := \nabla_{\nabla \rho} u_0 - \phi(x, u_0)$. Consequently, 
	\begin{equation}\label{C2_nn_ineq_1}
	\sup_{\partial M \times [0, T]} u_{\nu\nu} \leq C + \frac{1}{2} M,
	\end{equation}
	where the constant $C$ depends on $C_0$, $||d||_{C^3}$, $||\chi||_{C^1}$, $||u||_{C^{1, 1}}$, $||\phi||_{C^2}$, $||u_0||_{C^{3}}$, $|\psi|$, $|\nabla \psi|$, $k$, $n$ and the background geometric data.
\end{proposition}
\begin{proof}
	
	If we choose $\delta \leq  \frac{1}{\beta}$, then in $M_\delta$,
	\begin{equation}
	1 \leq 1 + \beta \rho \leq 1 + \beta \delta \leq 2.
	\end{equation}

	Since $\nabla \rho = \nu$ on the boundary $\partial M$, then it is easy to see that $W \equiv 0$ on $\partial M$. On $\partial M_{\delta} \setminus \partial M$, we have
	\begin{equation}
	W  \geq - C_1  + \frac{1}{2} A \delta \geq 0,
	\end{equation}
	if we choose $A \geq 2C_1 / \delta$, where the constant $C_1 = C_1 \left( || u ||_{C^1}, || d||_{C^1}, |\phi|, || u_0||_{C^1} \right)$.
	At $t = 0$, since $0 \leq d \leq \delta \leq \frac{1}{2}$, then $W = (A + \frac{M}{2}) d (1 - d) \geq 0$. Hence, $W \geq 0$ on the parabolic boundary $\partial M_{\delta, T}$. 
	
	Suppose $W$ attains a minimum at $(x_0, t_0) \in M_{\delta, T}$. 
	We choose an orthonormal frame $\{e_i\}$ around $x_0$ such that $\nabla_{e_i} e_j (x_0) = 0$ and the matrix $\left( U_{ij} \right)(x_0, t_0)$ is diagonalized. Then $F^{ij} := F^{ij}(U)(x_0, t_0)$ is also diagonalized. Then at $(x_0, t_0)$, we have
	\begin{equation}\label{C2_nn_1}
	0 \geq W_t = - (1 + \beta\rho) (\sum_k u_{tk} \rho_k - \phi_{u} u_t),
	\end{equation}
	\begin{equation}\label{C2_nn_2}
	\begin{aligned}
	0 = W_i &= - (1 + \beta\rho) \left(\sum_k u_{ki} \rho_k + \sum_k u_{k} \rho_{ki} - \phi_{i} - \phi_{u} u_{i} - q_{i} \right) \\
	&\qquad - \beta \rho_{i} (\sum_{k} u_{k}\rho_{k} - \phi - q) + (A + \frac{M}{2}) \rho_{i},
	\end{aligned}
	\end{equation}
	\begin{equation}\label{C2_nn_3}
	\begin{aligned}
	0 \leq W_{ii} &= - (1 + \beta\rho) \Bigg(\sum_k u_{kii} \rho_k + 2\sum_k u_{ki} \rho_{ki} + \sum_k u_{k} \rho_{kii} - \phi_{ii} - 2 \phi_{iu} u_{i}\\
	& \qquad - \phi_{u} u_{ii} - \phi_{uu} u_{i}^2 - q_{ii} \Bigg) 
	- \beta \rho_{ii} (\sum_{k} u_{k}\rho_{k} - \phi - q) + (A + \frac{M}{2}) \rho_{ii}\\
	&\qquad - 2 \beta\rho_{i} \left(\sum_k u_{ki} \rho_k + \sum_k u_{k} \rho_{ki} - \phi_{i} - \phi_{u} u_{i} - q_{i} \right),
	\end{aligned}
	\end{equation}
	and by (\ref{C2_ineq_7}),
	\begin{equation}\label{C2_nn_4}
	\begin{aligned}
	0 \leq \sum_i F^{ii} W_{ii} - W_t
	&\leq C_2 \beta \sum_i F^{ii} - (1 + \beta\rho) \sum_i ( 2\rho_{ii} - \phi_{u}) F^{ii} U_{ii}  \\
	&\qquad + (A + \frac{M}{2}) \sum_i F^{ii} \rho_{ii} - 2 \beta \sum_i F^{ii} U_{ii} \rho_i^2 ,
	\end{aligned}
	\end{equation}
	where $C_2 = C_2 \big(||d||_{C^3}, ||\chi||_{C^1}, |R_{ijkl}|, ||\phi||_{C^2}, |\psi|, |\nabla(\psi)|, || u||_{C^{1, 1}}, ||u_0||_{C^3} \big)$.

	Since $|\nabla \rho| \geq \frac{1}{2}$ in $M_{\delta}$, 
	then there exists $1 \leq k \leq n$ such that $\rho_{k}^2 \geq \frac{1}{4n}$. Without loss of generality, we assume $\rho_1^2 \geq \frac{1}{4n}.$ If we choose $\beta \geq 4n\gamma_1$, then the index set $J := \{ 1 \leq j \leq n : \beta \rho_j^2 \geq \gamma_1 \}$ is non-empty with $1 \in J$. For any $i\in J$, by (\ref{C2_nn_2}), we have
	\begin{equation}
	U_{ii} = \frac{1}{1 + \beta \rho} (A + \frac{M}{2}) + Q,
	\end{equation}
	where $$Q  := - \frac{1}{\rho_{i}} \left(-\sum_k \chi_{ki} \rho_k + \sum_k u_{k} \rho_{ki} - \phi_{i} - \phi_{u} u_{i} - q_{i} \right) 
	-  \frac{\beta}{1 + \beta \rho} (\sum_{k} u_{k}\rho_{k} - \phi - q).$$
	Since $\rho_{i}^2 \geq \frac{\gamma_1}{\beta}$, then $|Q| \leq C_3 \beta$, where $C_3 = C_3\big( |\chi_{ij}|, ||d||_{C^2}, || u||_{C^1}, ||\phi||_{C^1}, ||u_0||_{C^2}, \gamma_1  \big)$. 
	If we choose $A \geq 4 C_3 \beta$, that is, $\frac{A}{4} \geq C_3 \beta \geq |Q|$, then for $i\in J$,
	\begin{equation}\label{C2_nn_8}
	\frac{1}{4} A + \frac{1}{4}M \leq U_{ii} \leq \frac{5}{4} A + \frac{1}{2}M.
	\end{equation}
	Since $\rho_1^2 \geq \frac{1}{4n}$, and $U_{jj} \geq \frac{1}{4} A + \frac{1}{4}M > 0$ for $j \in J$, and $\beta\rho_{i}^2 \leq \gamma_1$ for $i \notin J$, and $2\rho_{kk} - \phi_u \leq -2\gamma_1$, then
	\begin{equation}\label{C2_nn_9}
	\begin{aligned}
	&\qquad - 2 \beta \sum_i F^{ii} U_{ii} \rho_i^2 - (1 + \beta\rho) \sum_i ( 2\rho_{ii} - \phi_{u}) F^{ii} U_{ii} \\
	&\leq - 2 \beta F^{11} U_{11} \rho_1^2 - 2 \sum_{i\notin J, U_{ii}<0} F^{ii} U_{ii} (\beta\rho_i^2) - (1 + \beta\rho) \sum_{U_{ii} < 0} ( 2\rho_{ii} - \phi_{u}) F^{ii} U_{ii}\\
	&\qquad - (1 + \beta\rho) \sum_{U_{ii} > 0} ( 2\rho_{ii} - \phi_{u}) F^{ii} U_{ii}\\
	&\leq - \frac{\beta}{2n} F^{11} U_{11} - 2 \gamma_1 \sum_{ U_{ii}<0} F^{ii} U_{ii} + 2 \gamma_1 \sum_{U_{ii} < 0} F^{ii} U_{ii} + 2 \sum_{U_{ii} > 0} | 2\rho_{ii} - \phi_{u}| F^{ii} U_{ii}\\
	&\leq - \frac{\beta}{2n} F^{11} U_{11} + \gamma_2 \sum_{U_{ii} > 0} F^{ii} U_{ii},
	\end{aligned}
	\end{equation}
	where $\gamma_2 := 4\gamma_0 + 2 \sup_{\overline{M_T}} |\phi_u|.$		
	Since $\nabla^2 \rho \leq - \gamma_1 g$ in $M_{\delta, T}$, then
	\begin{equation}\label{C2_nn_5}
	\begin{aligned}
	\sum_i F^{ii} \rho_{ii} \leq - \gamma_1 \sum_i F^{ii}.
	\end{aligned}
	\end{equation}
	Therefore, by (\ref{C2_nn_4}), (\ref{C2_nn_9}) and (\ref{C2_nn_5}), we have
	\begin{equation}\label{C2_nn_6}
	\begin{aligned}
	0 &\leq C_2 \beta \sum_i F^{ii} - (A + \frac{M}{2}) \gamma_1 \sum_i F^{ii} - \frac{\beta}{2n} F^{11} U_{11} + \gamma_2 \sum_{U_{ii} > 0} F^{ii} U_{ii}\\
	&\leq - \frac{\gamma_1 (A + M)}{2} \sum_i F^{ii} - \frac{\beta}{2n} F^{11} U_{11} + \gamma_2 \sum_{U_{ii} > 0} F^{ii} U_{ii},
	\end{aligned}
	\end{equation}
	where we choose $A \geq 2C_2\beta /\gamma_1$.

	Suppose $U_{11} \geq - \frac{3\gamma_2}{\gamma_1} \lambda_{min}$, then by (\ref{C2_nn_6}), (\ref{sigma_k_prop_2}), (\ref{sigma_k_prop_3}), (\ref{sigma_k_prop_8}) and (\ref{C2_nn_8}), we have the following contradiction
	\begin{equation}\label{C2_nn_7}
	\begin{aligned}
	0 &\leq - \frac{\gamma_1 (A + M)}{2} \sum_i F^{ii} - \frac{\beta}{2n} F^{11} U_{11} + \gamma_2 \sum_{U_{ii} > 0} F^{ii} U_{ii}\\
	&\leq  - \frac{\gamma_1 (A + M)}{2} \sum_i F^{ii} + \gamma_2 (k - \sum_{U_{ii} < 0} F^{ii} U_{ii}) \\
	&\leq - \frac{\gamma_1 (A + M)}{2} \sum_i F^{ii} + \gamma_2 k + \frac{\gamma_1}{3} U_{11} \sum_{i} F^{ii}\\
	&\leq - \frac{\gamma_1 (A + M)}{2} \sum_i F^{ii} + \gamma_2 k + \frac{\gamma_1}{3} (\frac{5}{4} A + \frac{1}{2}M) \sum_{i} F^{ii}\\ 
	&\leq - \frac{1}{12}\gamma_1 A C(n, k) e^{\frac{-1}{k}(u_t + \psi)} + \gamma_2 k - \frac{1}{3}\gamma_1 M\sum_i F^{ii} \\
	&\leq - \frac{1}{3}\gamma_1 M\sum_i F^{ii},
	\end{aligned}
	\end{equation}
	where we choose $A \geq \frac{12\gamma_2 k}{\gamma_1 C(n,k)} \sup e^{\frac{1}{k}(u_t + \psi)}$.

	Suppose $U_{11} \leq - \frac{3\gamma_2}{\gamma_1} \lambda_{min}$. By (\ref{C2_nn_0}) and (\ref{C2_nn_8}),
	\begin{equation}\label{C2_nn_10}
	\lambda_{max} \leq C_0 (1 + M) \leq 4 C_0 U_{11},
	\end{equation}
	if we choose $A \geq 1$. By Lemma \ref{sigma_k_lemma}, we have 
	\begin{equation}\label{C2_nn_12}
	F^{11} \geq \frac{1}{C_3} \sum_i F^{ii},
	\end{equation} 
	where $\frac{1}{C_3} := \left( \frac{\gamma_1}{12 C_0 \gamma_2} \right)^2 \frac{k - 1}{(n-1)(n-2+k) (n-k+1)}$.
	By (\ref{C2_nn_6}), (\ref{C2_nn_8}), (\ref{C2_nn_10}) and (\ref{C2_nn_12}), we have the following contradiction
	\begin{equation}\label{C2_nn_11}
	\begin{aligned}
	0 &\leq - \frac{\gamma_1 (A + M)}{2} \sum_i F^{ii} - \frac{\beta}{2n} F^{11} U_{11} + \gamma_2 \sum_{U_{ii} > 0} F^{ii} U_{ii}\\
	&\leq - \frac{\gamma_1 (A + M)}{2} \sum_i F^{ii} - \frac{\beta}{8n C_3} (A + M) \sum_i F^{ii} + \gamma_2 C_0 (1 + M) \sum_{i} F^{ii}\\ &\leq - \frac{\gamma_1 (A + M)}{2} \sum_i F^{ii},
	\end{aligned}
	\end{equation}
	where we choose $\beta \geq 8n\gamma_2 C_0 C_3$ and $A \geq 1$.

	In conclusion, if we choose $\beta = 8 n \gamma_2 C_0 C_3 \geq 2$, as well as $\delta = \frac{1}{\beta}$, and $A = \max\{2C_1/\delta, 2C_2 \beta / \gamma_1, \frac{12\gamma_2 k}{\gamma_1 C(n,k)} \sup e^{\frac{1}{k}(u_t + \psi)}, 1\}$, then $W$ attains its minimum only on $\partial M_{\delta, T}$, and consequently, $W \geq 0$ in $\overline{M_{\delta, T}}$.
	
	Suppose $\sup_{\partial M \times [0, T]} u_{\nu\nu}$ is attained at $(x_1, t_1) \in \partial M \times [0, T]$, then at $(x_1, t_1)$, we have
	\begin{equation}
	\begin{aligned}
	0 &\leq W_\nu 
	= - \left(u_{\nu\nu} + \sum_k u_{k} \rho_{k\nu} - \phi_\nu - \phi_{u} \phi - q_{\nu} \right) + (A + \frac{M}{2}) \\
	&\leq - u_{\nu\nu} + C + \frac{M}{2},
	\end{aligned}
	\end{equation}
	and hence (\ref{C2_nn_ineq_1}) holds. 
\end{proof}

\begin{proposition}\label{C2_nn_inf_u_nn}
	There exist large $A, \beta$ and small $\delta$ such that
	\begin{equation}\label{C2_nn_ineq_2}
	\underline{W} := - (1 + \beta \rho) \left( \nabla_{\nabla \rho} u - \phi - q \right) - (A + \frac{M}{2}) \rho \leq 0 \qquad \text{in} \; \overline{M_{\delta, T}},
	\end{equation}
	where $q(x) := \nabla_{\nabla \rho} u_0 - \phi(x, u_0)$. Consequently, 
	\begin{equation}
	\inf_{\partial M \times [0, T]} u_{\nu\nu} \geq - C - \frac{1}{2} M,
	\end{equation}
	where the constant $C$ depends on $C_0$, $||d||_{C^3}$, $||\chi||_{C^1}$, $||u||_{C^{1, 1}}$, $||\phi||_{C^2}$, $||u_0||_{C^{3}}$, $|\psi|$, $|\nabla \psi|$, $k$, $n$ and the background geometric data.
\end{proposition}
\begin{proof}
	Similar to Proposition \ref{C2_nn_sup_u_nn}, if we choose $\delta \leq \frac{1}{\beta}$ and $A \geq 2C_1 / \delta$, where the constant $C_1 =  C_1 \left( || u ||_{C^1}, || d||_{C^1}, |\phi|, || u_0||_{C^1} \right) \geq |\nabla_{\nabla \rho} u - \phi - q|$, then we have $1 \leq 1 + \beta \rho \leq 2$ and $\underline{W} \leq 0$ on the parabolic boundary $\partial M_{\delta, T}$.

	Suppose $\underline{W}$ attains a maximum at $(x_0, t_0) \in M_{\delta, T}$. 
	We choose an orthonormal frame $\{e_i\}$ around $x_0$ such that $\nabla_{e_i} e_j (x_0) = 0$ and the matrix $\left( U_{ij} \right)(x_0, t_0)$ is diagonalized. Then $F^{ij} := F^{ij}(U)(x_0, t_0)$ is also diagonalized. Then at $(x_0, t_0)$, similar to (\ref{C2_nn_4}), we have
	\begin{equation}\label{C2_nn_24}
	\begin{aligned}
	0 \geq \sum_i F^{ii} \underline{W}_{ii} - \underline{W}_t
	&\geq - C_2 \beta \sum_i F^{ii} - (1 + \beta\rho) \sum_i ( 2\rho_{ii} - \phi_{u}) F^{ii} U_{ii}  \\
	&\qquad - (A + \frac{M}{2}) \sum_i F^{ii} \rho_{ii} - 2 \beta \sum_i F^{ii} U_{ii} \rho_i^2 ,
	\end{aligned}
	\end{equation}
	where $C_2 = C_2 \big(||d||_{C^3}, ||\chi||_{C^1}, |R_{ijkl}|, ||\phi||_{C^2}, |\psi|, |\nabla(\psi)|, ||u||_{C^1}, |u_t|, ||u_0||_{C^3} \big)$.

	Since $|\nabla \rho| \geq \frac{1}{2}$ in $M_{\delta}$, 
	without loss of generality, we assume $\rho_1^2 \geq \frac{1}{4n}.$ 
	If we choose $\beta \geq 4n\gamma_1$, then the index set $J := \{ 1 \leq j \leq n : \beta \rho_j^2 \geq \gamma_1 \}$ is non-empty with $1 \in J$. For any $i\in J$, we have
	\begin{equation}
	U_{ii} = - \frac{1}{1 + \beta \rho} (A + \frac{M}{2}) + Q,
	\end{equation}
	where $|Q| \leq C_3 \beta$ and $C_3 = C_3\big( |\chi_{ij}|, ||d||_{C^2}, |\nabla u|, ||\phi||_{C^1}, ||u_0||_{C^2}, \gamma_1  \big)$. 
	If we choose $A \geq 4 C_3 \beta$, that is, $\frac{A}{4} \geq C_3 \beta \geq |Q|$, then for $i\in J$,
	\begin{equation}\label{C2_nn_28}
	-\frac{1}{4} A - \frac{1}{4}M \geq U_{ii} \geq - \frac{5}{4} A - \frac{1}{2}M.
	\end{equation}
	Since $\rho_1^2 \geq \frac{1}{4n}$, and $U_{jj} \leq -\frac{1}{4} A - \frac{1}{4}M < 0$ for $j \in J$, and $\beta\rho_{i}^2 \leq \gamma_1$ for $i \notin J$, and $2\rho_{kk} - \phi_u \leq -2\gamma_1$, then
	\begin{equation}\label{C2_nn_29}
	\begin{aligned}
	&\qquad - 2 \beta \sum_i F^{ii} U_{ii} \rho_i^2 - (1 + \beta\rho) \sum_i ( 2\rho_{ii} - \phi_{u}) F^{ii} U_{ii} \\
	&\geq - 2 \beta F^{11} U_{11} \rho_1^2 - 2 \sum_{i\notin J, U_{ii}>0} F^{ii} U_{ii} (\beta \rho_i^2) - (1 + \beta\rho) \sum_{U_{ii} > 0} ( 2\rho_{ii} - \phi_{u}) F^{ii} U_{ii}\\
	&\qquad - (1 + \beta\rho) \sum_{U_{ii} < 0} ( 2\rho_{ii} - \phi_{u}) F^{ii} U_{ii}\\
	&\geq - \frac{\beta}{2n} F^{11} U_{11} - 2 \gamma_1 \sum_{ U_{ii}>0} F^{ii} U_{ii} + 2 \gamma_1 \sum_{U_{ii} > 0} F^{ii} U_{ii} + 2 \sum_{U_{ii} < 0} | 2\rho_{ii} - \phi_{u}| F^{ii} U_{ii}\\
	&\geq - \frac{\beta}{2n} F^{11} U_{11} + \gamma_2 \sum_{U_{ii} < 0} F^{ii} U_{ii},
	\end{aligned}
	\end{equation}
	where $\gamma_2 := 4\gamma_0 + 2 \sup_{\overline{M_T}} |\phi_u|.$		
	Therefore, by (\ref{C2_nn_24}), (\ref{C2_nn_29}) and (\ref{C2_nn_5}), we have
	\begin{equation}\label{C2_nn_26}
	\begin{aligned}
	0 &\geq - C_2 \beta \sum_i F^{ii} + (A + \frac{M}{2}) \gamma_1 \sum_i F^{ii} - \frac{\beta}{2n} F^{11} U_{11} + \gamma_2 \sum_{U_{ii} < 0} F^{ii} U_{ii}\\
	&\geq \frac{\gamma_1 (A + M)}{2} \sum_i F^{ii} - \frac{\beta}{2n} F^{11} U_{11} + \gamma_2 \sum_{U_{ii} < 0} F^{ii} U_{ii},
	\end{aligned}
	\end{equation}
	where we choose $A \geq 2C_2\beta /\gamma_1$.

	Since $U_{11} < 0$, as in (\ref{sigma_k_lemma_2}), we have
	\begin{equation}
	F^{11} \geq \frac{1}{n-k+1} \sum_i F^{ii}.
	\end{equation} 
	By (\ref{C2_nn_26}), (\ref{C2_nn_28}) and (\ref{C2_nn_0}), we have the following contradiction
	\begin{equation}\label{C2_nn_31}
	\begin{aligned}
	0 &\geq \frac{\gamma_1 (A + M)}{2} \sum_i F^{ii} - \frac{\beta}{2n} F^{11} U_{11} + \gamma_2 \sum_{U_{ii} < 0} F^{ii} U_{ii}\\
	&\geq \frac{\gamma_1 (A + M)}{2} \sum_i F^{ii} + \frac{\beta(A + M)}{8n(n-k+1)} \sum_i F^{ii} - \gamma_2 C_0 (1 + M) \sum_{i} F^{ii}\\
	& \geq \frac{\gamma_1 (A + M)}{2} \sum_i F^{ii},
	\end{aligned}
	\end{equation}
	where we choose $\beta \geq 8n(n-k+1)\gamma_2 C_0$ and $A \geq 1$.

	In conclusion, if we choose $\beta = 8 n (n-k+1) \gamma_2 C_0 \geq 2$, and $\delta = \frac{1}{\beta}$, and $A = \max\{2C_1/\delta, 2C_2 \beta / \gamma_1 \} \geq 1$, then $\underline{W}$ attains its maximum only at $\partial M_{\delta, T}$, and consequently, $\underline{W} \leq 0$ in $\overline{M_{\delta, T}}$.
	
	Suppose $\inf_{\partial M \times [0, T]} u_{\nu\nu}$ is attained at $(x_1, t_1) \in \partial M \times [0, T]$, then at $(x_1, t_1)$, we have
	\begin{equation}
	\begin{aligned}
	0 &\geq \underline{W}_\nu 
	= - \left(u_{\nu\nu} + \sum_k u_{k} \rho_{k\nu} - \phi_\nu - \phi_{u} \phi - q_{\nu} \right) - (A + \frac{M}{2}) \\
	&\geq - u_{\nu\nu} - C - \frac{M}{2},
	\end{aligned}
	\end{equation}
	and hence (\ref{C2_nn_ineq_2}) holds. 
\end{proof}

\section{Long-time existence and uniform convergence theorem}\label{conclusion.sec}
\begin{theorem}\label{sigma_k_long_time}
	Let $(\overline{M}, g)$ be a compact Riemanian manifold with non-negative sectional curvatures and uniformly strictly convex boundary $\partial M$\footnote{The principal curvatures of $\partial M$ is bounded below by a positive constant $\underline{\kappa}$.} and $\chi$ a smooth (0, 2)-tensor on $\overline{M}$. 
	Suppose $\psi(x, z, t) \in C^{\infty}\big( \overline{M}\times\mathbb{R}^2 \big)$ with either $\psi_z \equiv 0$ or $\psi_t \equiv 0$, and $\phi(x, z) \in C^{\infty} \big(\partial M \times \mathbb{R}\big)$ with $\phi_z \geq 0$, and $u_0 \in C^{\infty}(\overline{M})$ with $(u_0)_\nu = \phi(x, u_0)$ on $\partial M$ and $\lambda_g(\chi + \nabla^2 u_0) \in \Gamma_k(\mathbb{R}^n)$.
	Then the IBV problem
	\begin{equation}\label{eq_IBV_k-Hessian}
	\left\{
	\begin{aligned}
	u_t &= \log \sigma_k\big(\lambda_g(\chi_{ij} + u_{ij})\big) - \psi(x, u, t) \qquad &\text{in}&\; M \times \{t>0\}, \\ 
	u_\nu &= \phi(x, u) \qquad &\text{on}& \; \partial M \times \{t \geq 0\},\\
	u  &= u_0 \qquad & \text{in}& \; \overline{M}\times\{t = 0\},
	\end{aligned}
	\right.
	\end{equation}
	has a unique smooth solution $u \in C^{\infty} \big(\overline{M} \times (0, \infty) \big)$. 
\end{theorem}

\begin{remark}
	If we want to obtain smoothness of $u$ at $t = 0$, we can consider the following compatibility condition on $\partial M$
	\begin{equation}
	\nabla_\nu \left[ F\big(\chi_{ij} + (u_0)_{ij}\big) - \psi(x, u_0, 0) \right] = \phi_u(x, u_0) \left[ F\big(\chi_{ij} + (u_0)_{ij}\big) - \psi(x, u_0, 0) \right].
	\end{equation}
\end{remark}

\begin{proof}
	
	Since $u_0$ satisfies the zeroth-order compatibility condition $(u_0)_\nu = \phi(x, u_0)$ on $\partial M$, then short-time existence and uniqueness of solutions are well known from the standard theories.
	
	The long-time existence depends on a priori estimates. 
	The admissible function $\underline{u}(x, t) := u_0(x)$ is a (parabolic) $\mathcal{C}$-subsolution for the parabolic equation (\ref{eq_IBV_k-Hessian}) (see Corollary 2.2.5 in \cite{GuoThesis19}). 
	By Lemma \ref{u_t_bounds}, Theorem \ref{grad_estimate}, Theorem \ref{C2_interior_estimate} and Theorem \ref{sigma_k_u_nn_estimates}, we obtain $C^{2, 1}$ estimates on $\overline{M_T}$ for any $T > 0$. By Evans-Krylov theorem \cite{Evans82} \cite{Kry82} and Schauder estimates (see \cite{LSU68}), we obtain higher order estimates on $\overline{M} \times [\epsilon, T]$ for any $0 < \epsilon < T$. Hence, $u \in C^{\infty}\big( \overline{M} \times (0, \infty) \big)$.
\end{proof}

Theorem \ref{elliptic_Neumann_sigma_k} is a corollary of the following theorem when $\rho \equiv 1$.

\begin{theorem}\label{elliptic_Neumann_sigma_k_general}
	Let $(\overline{M}, g)$ be a compact Riemanian manifold with non-negative sectional curvatures and uniformly strictly convex boundary $\partial M$ and $\chi$ a smooth (0, 2)-tensor on $\overline{M}$.
	Suppose $\phi(x) \in C^{\infty} \big(\partial M\big)$, and there exists $\underline{u} \in C^{\infty}(\overline{M})$ such that $\lambda_g(\chi + \nabla^2 \underline{u}) \in \Gamma_k(\mathbb{R}^n)$, and $\underline{u}_\nu = \phi(x)$ on $\partial M$.
	Then for any $\psi, \rho \in C^{\infty}\big( \overline{M} \big)$ with $|\rho| > 0$, there exists a constant $c$ such that the Neumann boundary problem
	\begin{equation}\label{univ_conv_sigma_k_eq}
	\left\{
	\begin{aligned}
	&\sigma_k\big(\lambda_g(\chi_{ij} + u_{ij})\big) = e^{\psi + c\rho} \qquad &\text{in}&\; M, \\ 
	&u_\nu = \phi(x) \qquad &\text{on}& \; \partial M,   
	\end{aligned}
	\right.
	\end{equation}
	has a unique smooth solution $u \in C^{\infty} \big(\overline{M} \big)$ up to a constant.
\end{theorem}

\begin{remark}
	When $k=1$, it is a uniformly elliptic equation, and the curvatures conditions of $\overline{M}$ can be removed (see Theorem 5.1.5 in \cite{GuoThesis19}).
\end{remark}

\begin{proof}
	Without loss of generality, assume $\rho > 0$.
	Let $F(\chi_{ij} + u_{ij}) := \log \sigma_k\big(\lambda_g(\chi_{ij} + u_{ij})\big) / \rho$ and $\tilde{\psi} := \psi / \rho$.
	By Theorem \ref{sigma_k_long_time}, the IBV problem
	\begin{equation}
	\left\{
	\begin{aligned}
	u_t &= F(\chi_{ij} + u_{ij}) - \tilde{\psi}(x) \qquad &\text{in}&\; M \times \{t>0\}, \\ 
	u_\nu &= \phi(x) \qquad &\text{on}& \; \partial M \times \{t \geq 0\},\\
	u  &= \underline{u} \qquad & \text{in}& \; \overline{M}\times\{t = 0\},
	\end{aligned}
	\right.
	\end{equation}
	has a smooth solution $u \in C^{\infty}\big(\overline{M}\times (0, \infty)\big)$. By Lemma \ref{u_t_bounds}, $u_t$ is bounded, and hence $F(\chi_{ij} + u_{ij}) = u_t + \tilde{\psi}(x) \leq C_1$, where $C_1$ is independent of $t$. The admissible function $\underline{u}$ is a $\mathcal{C}$-subsolution of $F(\chi_{ij} + u_{ij}) = C_1$ (see Proposition 2.2.4 in \cite{GuoThesis19}). Therefore, by Theorem \ref{osc_thm}, 
	\begin{equation}\label{unif_conv_osc}
	\sup_{t\geq 0} \Big(\sup_{\overline{M}} u(\cdot, t) - \inf_{\overline{M}} u(\cdot, t)\Big) \leq C.
	\end{equation}
	
	Fix $x_0 \in M$, and let $$v(x,t) := u(x, t) - u(x_0, t), $$
	then $v(x, t)$ satisfies the following IBV problem
	\begin{equation}\label{unif_conc_v_eq}
	\left\{
	\begin{aligned}
	v_t &= F(\chi_{ij} + v_{ij}) - \tilde{\psi}(x) - u_t(x_0, t) \qquad &\text{in}&\; M \times \{t>0\}, \\ 
	v_\nu &= \phi(x) \qquad &\text{on}& \; \partial M \times \{t \geq 0\},\\
	v &= \underline{u} - u(x_0, 0) \qquad & \text{in}& \; \overline{M}\times\{t = 0\}.
	\end{aligned}
	\right.
	\end{equation}
	By (\ref{unif_conv_osc}), we have $\sup_{t\geq 0} ||v(x, t)||_{C^0(\overline{M})} \leq C$.
	By Lemma \ref{u_t_bounds}, 
	$\sup_{\overline{M}\times [0, \infty)} |v_t| \leq 2 \sup_{\overline{M}\times [0, \infty)} |u_t| \leq 2\sup_{\overline{M}\times \{t=0\}} |u_t| \leq C.$
	By Theorem \ref{grad_estimate}, $\sup_{t\geq 0} ||v||_{C^{1} (\overline{M})} \leq C$. 
	The admissible function $\underline{v}(x, t) := \underline{u}(x)$ is a (parabolic) $\mathcal{C}$-subsolution of the parabolic equation (\ref{unif_conc_v_eq}) (see Corollary 2.2.5 in \cite{GuoThesis19}) with $|\underline{v}|, |\underline{v}_t|$ and $\lambda_g(\chi + \nabla^2 \underline{v})$ being bounded. 
	By Theorem \ref{C2_interior_estimate} and Theorem \ref{sigma_k_u_nn_estimates}, we have
	$\sup_{t\geq 0} ||v||_{C^{2} (\overline{M})} \leq C$\footnote{$C^2$ estimates does not depend on $\partial_t \big(\psi(x) + u_t(x_0, t)\big)$, that is, $u_{tt}(x_0, t)$.}, 
	and consequently, $\theta g^{ij} \leq F^{ij} \leq \frac{1}{\theta} g^{ij}$ for some constant $\theta \in (0, 1]$ independent of $t$, where $F^{ij} := F^{ij} (\chi_{ij} + u_{ij})$. Fix $\epsilon > 0$ small, by Evans-Krylov theorem and Schauder estimates, we have $||v||_{C^{4+\alpha, 2+\frac{1}{2}\alpha} \big(\overline{M} \times [\epsilon, \infty)\big)} \leq C,$ 
	and consequently, $ ||F^{ij}||_{C^{2, 1}\big(\overline{M} \times [\epsilon, \infty)\big)} \leq C$.
	
	Since $F^{ij}(\chi_{ij} + u_{ij}) = F^{ij}(\chi_{ij} + v_{ij})$, then $u_t$ satisfies the following linear parabolic equation with Neumann boundary condition
	\begin{equation}\label{du/dt_parabolic}
	\left\{
	\begin{aligned}
	w_t &= a^{ij}w_{ij} \qquad &\text{in}&\; M\times [\epsilon, \infty), \\ 
	w_\nu &= 0 \qquad &\text{on}& \; \partial M \times [\epsilon, \infty),  
	\end{aligned}
	\right.
	\end{equation}
	where $a^{ij} := F^{ij}(\chi_{ij} + v_{ij})$.	
	We can use Cao's method \cite{Cao85} and Theorem \ref{Harnack_ineq_thm} to conclude that $u_t \rightarrow c$ uniformly for some constant $c$, and $v_t \rightarrow 0$ uniformly. 
	
	Since $\sup_{t\geq\epsilon}||v(x, t)||_{C^{2+\alpha} (\overline{M})} \leq C$, then there exists a sequence $t_n \rightarrow \infty$ such that the sequence $\{v(x, t_n)\}$ converges to some function $\tilde{u}(x)$ in $C^{2+\alpha}(\overline{M})$, and $\tilde{u}$ satisfies the equation (\ref{univ_conv_sigma_k_eq}).
	By Schauder estimates (see \cite{Lieberman13}), $\tilde{u} \in C^{\infty}(\overline{M})$.
	
	Uniqueness can be derived by maximum principles and the Hopf lemma. In fact, suppose there are two different solutions $u_1$ and $u_2$ with corresponding constants $c_1 \geq c_2$ for the Neumann problem (\ref{univ_conv_sigma_k_eq}). Let $w := u_1 - u_2$, and $\tilde{a}^{ij} := \int^1_0 F^{ij}\Big( \chi_{ij} + t (u_1)_{ij} + (1-t) (u_2)_{ij}\Big) dt$, then $w$ satisfies the following Neumann problem
	\begin{equation}
	\left\{
	\begin{aligned}
	&\tilde{a}^{ij}w_{ij} = (e^{c_1} - e^{c_2}) e^{\tilde{\psi}(x)} \qquad &\text{in}&\; M, \\ 
	&w_\nu = 0 \qquad &\text{on}& \; \partial M.  
	\end{aligned}
	\right.
	\end{equation}
	By weak maximum principle, $w$ attains a maximum at $x_0 \in \partial M$. Suppose for all $x\in M$, $w(x) < w(x_0)$, then by the Hopf lemma, $w_\nu (x_0) > 0$, a contradiction. Hence, $w$ is a constant and $c_1 = c_2$ by strong maximum principle.
\end{proof}

Similarly, we have the following theorem for linear elliptic equations.
\begin{theorem}
	Let $(\overline{M}, g)$ be a compact Riemanian manifold with smooth boundary $\partial M$.
	Let $\mathcal{L} := a^{ij}\nabla_i \nabla_j$ be a linear uniformly elliptic operator on $\overline{M}$. 
	Then for any $\phi, \psi, \rho \in C^{\infty}\big( \overline{M} \big)$ with $|\rho| > 0$, there exists a constant $c$ such that the following Neumann boundary problem
	\begin{equation}\label{eq_linear_elliptic_2}
	\left\{
	\begin{aligned}
	&\mathcal{L} u = \psi(x) + c \rho \qquad &\text{in}&\; M, \\ 
	&u_\nu = \phi(x) \qquad &\text{on}& \; \partial M,   
	\end{aligned}
	\right.
	\end{equation}
	has a unique smooth solution $u \in C^{\infty} \big(\overline{M} \big)$ up to a constant. Especially, when $\mathcal{L} = \Delta_g$, we have $c = - \frac{1}{\int_{\overline{M}} \rho \,dV} \Big( \int_{\overline{M}} \psi \,dV + \int_{\partial M} \phi \,dS \Big).$
\end{theorem}
\begin{proof}
	The proof is similar to Theorem \ref{elliptic_Neumann_sigma_k_general}.
	The a priori estimates for linear equations with Neumann boundary condition are well known (see \cite{LSU68}).
	The admissible function $\underline{u} := \phi d \in C^{\infty}(\overline{M})$ with $\underline{u}_{\nu} = \phi(x)$ on $\partial M$ is a $\mathcal{C}$-subsolution for the elliptic equation $\mathcal{L} u = \psi(x)$, where $d$ is defined as in (\ref{dist2bd_function}) (see Proposition 2.2.8 in \cite{GuoThesis19}). 
\end{proof}

\begin{example}
	For any $\psi \in C^{\infty}\big( \overline{\mathbb{D}^{n}} \big)$, there exists a constant $c$ such that the Neumann boundary problem
	\begin{equation}
	\left\{
	\begin{aligned}
	&\det\big(u_{ij}\big) = e^{\psi(x) + c} \qquad &\text{in}&\; \mathbb{D}^{n}, \\ 
	&u_\nu = -1 \qquad &\text{on}& \; \mathbb{S}^{n-1},   
	\end{aligned}
	\right.
	\end{equation}
	has a unique smooth solution $u \in C^{\infty} \big(\overline{M} \big)$ up to a constant.
\end{example}
\begin{proof}
	We construct a $\mathcal{C}$-subsolution $\underline{u} := \frac{1}{2} |x|^2$ with $\underline{u}_\nu = -1$ on $\overline{\mathbb{D}^{n}}$, and the above conclusion follows immediately after Theorem \ref{elliptic_Neumann_sigma_k}. 
\end{proof}

\appendix
\section{Parabolic Harnack Inequalities}
\label{Harnack.app}

Harnack inequalities for Heat equations with Neumann boundary conditions on compact Riemannian manifolds with smooth boundary were studied by Li-Yau \cite{LiYau86}, Yau \cite{Yau95} and Chen \cite{ChenR00}.
In the following, we derive a Harnack inequality for general linear parabolic equations with vanishing Neumann boundary condition. The proof is motivated by Theorem 10 of Section 7.1 in \cite{EvansBk}.

\begin{theorem}\label{Harnack_ineq_thm}
	Let $(\overline{M}, g)$ be a smooth connected Riemannian manifolds with bounded sectional curvatures $|\text{Sec}| \leq K $ in $\overline{M}$ and bounded principal curvatures $|I\!\!I| \leq H$ on $\partial M$. Let $u \in C^{2, 1}(\overline{M_T})$ be the nonnegative solution of the following linear uniformly parabolic equation with Neumann boundary condition
	\begin{equation}\label{eq_Harnack}
	\left\{
	\begin{aligned}
	u_t &= a^{ij}u_{ij} \qquad &\text{in}\; M_T, \\ 
	u_\nu &= 0 \qquad &\text{on} \; \partial M,   \\
	u &\geq 0 \qquad &\text{in} \; \overline{M_T},
	\end{aligned}
	\right.
	\end{equation}
	where $a^{ij} \in C^{2, 1}(\overline{M_T})$ and $\theta( g^{ij}) \leq (a^{ij}) \leq \frac{1}{\theta} (g^{ij})$ with $0 < \theta \leq 1$. 
	Let $\nu := \nabla d$ in $\overline{M}$ with $d$ defined in (\ref{dist2bd_function}).
	Then there exists a constant $\gamma > 1$ such that for $0 < t_1 < t_2 \leq T$,
	\begin{equation}\label{Harnack's_ineq}
	\sup_{\overline{M}} u(\cdot, t_1) \leq \gamma \inf_{\overline{M}} u(\cdot, t_2),
	\end{equation}
	where $\gamma$ depends on $\theta$, $t_1, t_2$, $K, H$,  $\text{diam}(\overline{M})$, $||\nabla_\nu \nu||_{C^0(\partial M)}$, $||a^{ij}||_{C^{2, 1}(\overline{M_T})}$,  $||d||_{C^2(\overline{M})}$.
\end{theorem}
\begin{remark}
	If $a^{ij} = g^{ij}$, then $\gamma$ depends on the lower bounds of Ricci curvatures $R_{ij} \geq -\underline{K} g_{ij} \quad (\underline{K} \geq 0)$ and the lower bounds of principal curvatures $I\!\!I \geq -\underline{H} \quad (\underline{H} \geq 0)$ instead of $K$ and $H$, and it is independent of $||\nabla_\nu \nu||_{C^0(\partial M)}$. This case was studied by Chen \cite{ChenR00}.
\end{remark}

\begin{proof}
	Let $$v:= \log (u + \epsilon_0)$$ with $\epsilon_0 > 0$. Then by the parabolic equation (\ref{eq_Harnack}), we have
	\begin{equation}\label{v_t}
	v_t = w + \tilde{w},
	\end{equation}
	where 
	\begin{equation}\label{Harn_w&tilde_w}
	w := a^{ij}v_{ij}, \qquad \tilde{w} := a^{ij} v_i v_j.
	\end{equation}

	Fix $x \in \overline{M}$, and choose an orthonormal frame $\{e_i\}$ near $x$ such that $\nabla_{e_i}e_j (x) = 0$ for $1 \leq i, j \leq n$.
	By (\ref{Harn_w&tilde_w}) and (\ref{ijk-ikj}),
	\begin{equation}\label{Harn_ineq_1}
	\begin{aligned}
	a^{kl} \tilde{w}_{kl}
	=& - b^k a^{ij} v_{ijk} + 2 a^{kl} a^{ij} R^m_{kil} v_m v_j + 2 a^{kl} a^{ij} v_{ik} v_{jl} \\
	&+ 4 a^{kl} a^{ij}_k v_{il} v_j + a^{kl} a^{ij}_{kl} v_i v_j,
	\end{aligned}	
	\end{equation}
	where
	\begin{equation}\label{Harn_ineq_2}
	b^k := -2 a^{kl} v_l.
	\end{equation}
	Since $a^{ij}v_{ijk} = w_k - a^{ij}_k v_{ij}$, then
	by (\ref{v_t}) (\ref{Harn_w&tilde_w}) (\ref{Harn_ineq_1}) and Cauchy-Schwarz inequality,
	\begin{equation}\label{w_t-Lw}
	\begin{aligned}
	w_t - a^{kl} w_{kl} + b^k w_k &= a^{kl}_t v_{kl} + a^{kl} v_{klt} - a^{kl} w_{kl} + b^k w_k \\
	&\geq 2\theta^2 |\nabla^2 v|^2 - \epsilon|\nabla^2 v|^2 - \frac{C}{\epsilon} |\nabla v|^2 - \frac{C}{\epsilon} + 2 a^{kl} a^{ij} R^m_{kil} v_m v_j\\
	&\geq \frac{3}{2}\theta^2 |\nabla^2 v|^2 - C |\nabla v|^2 - C + 2 a^{kl} a^{ij} R^m_{kil} v_m v_j,
	\end{aligned}
	\end{equation}
	where we choose $\epsilon = \frac{ \theta^2}{2}$.
	By (\ref{v_t}) (\ref{Harn_w&tilde_w}) (\ref{Harn_ineq_1}) and Cauchy-Schwarz inequality,
	\begin{equation}\label{tilde_w_t-Lw}
	\begin{aligned}
	\tilde{w}_t - a^{kl} \tilde{w}_{kl} + b^k \tilde{w}_k 
	&\geq - 2 a^{ij} a^{kl} v_{ik} v_{jl} - \epsilon |\nabla^2 v|^2 - \frac{C}{\epsilon} |\nabla v|^2 - 2 a^{kl} a^{ij} R^m_{kil} v_m v_j\\
	&\geq - (\frac{2}{\theta^2} + \epsilon) |\nabla^2 v|^2 - \frac{C}{\epsilon} |\nabla v|^2 - 2 a^{kl} a^{ij} R^m_{kil} v_m v_j\\
	&\geq - \frac{4}{\theta^2} |\nabla^2 v|^2 - C |\nabla v|^2 - 2 a^{kl} a^{ij} R^m_{kil} v_m v_j,
	\end{aligned}
	\end{equation}
	where we choose $\epsilon = \frac{2}{\theta^2}$.

	To deal with the term $2 a^{kl} a^{ij} R^m_{kil} v_m v_j$ above, we can rearrange the orthonormal frame $\{e_i\}$ so that $a^{ij}$ is diagonalized at $x$. Since $|\text{Sec}| \leq K$, then $O \leq (R_{mkik} + K \delta_{mi}) \leq 2K(\delta_{mi})$\footnote{Fix $k$, $(M_{mi}) := (R_{mkik})$ is a symmetric $n\times n$ matrix.} for each $k$, and therefore by Cauchy-Schwarz inequality,
	\begin{equation}\label{sec_curv_term_1}
	\begin{aligned}
	&\qquad 2 a^{kl} a^{ij} R^m_{kil} v_m v_j
	= 2 \sum_{i, m, k} a^{kk} a^{ii} R_{mkik} v_m v_i \\
	&= 2 \sum_{i, m, k} a^{kk} a^{ii} (R_{mkik} + K \delta_{mi}) v_m v_i -2 \sum_{i, k} a^{kk} a^{ii}  K v_i v_i \\
	&\geq - \sum_{i, m, k} a^{kk} (R_{mkik} + K \delta_{mi})[(a^{mm}v_m)(a^{ii} v_i) + v_m v_i] - \frac{2nK}{\theta^2} |\nabla v|^2 \\
	& \geq - \frac{4nK}{\theta^3} |\nabla v|^2 - \frac{2nK}{\theta^2} |\nabla v|^2.
	\end{aligned}
	\end{equation}
	
	Let $$W := w + (\kappa - Ad)\tilde{w},$$ 
	where $\kappa \leq 1$ (small) and $A$ (large) are positive constants to be determined later. 
	We can choose $\delta_0$ small enough (recall that $0 \leq d \leq \delta_0$) so that $\kappa - Ad \geq \frac{\kappa}{2}$. Then by (\ref{w_t-Lw}), (\ref{tilde_w_t-Lw}) and (\ref{sec_curv_term_1}),
	\begin{equation}\label{Cap_w_t-Lw}
	\begin{aligned}
	&\qquad W_t - a^{kl} W_{kl} + b^k W_k \\
	&= (w_t - a^{kl} w_{kl} + b^k w_k) + (\kappa - Ad) (\tilde{w}_t - a^{kl} \tilde{w}_{kl} + b^k \tilde{w}_k) \\
	& \quad + a^{kl} (2 A d_k \tilde{w}_l + A d_{kl} \tilde{w}) -A b^k d_k \tilde{w} \\
	&\geq \Big(\frac{3}{2}\theta^2 |\nabla^2 v|^2 - C |\nabla v|^2 - C\Big) + \kappa \Big(- \frac{4}{\theta^2} |\nabla^2 v|^2 - C |\nabla v|^2\Big) \\
	&\quad - (\epsilon|\nabla^2 v|^2 + \frac{C}{\epsilon} |\nabla v|^2 + C |\nabla v|^3) + 2 [1 - (\kappa - Ad)] a^{kl} a^{ij} R^m_{kil} v_m v_j\\
	& \geq \theta^2 |\nabla^2 v|^2 - C|\nabla v|^3 - C |\nabla v|^2 - C,
	\end{aligned}
	\end{equation}
	where we choose $\kappa \leq \frac{\theta^4}{16}$ and $\epsilon = \frac{\theta^2}{4}$.
	
	Let $$Z := \zeta^2 W + \mu t,$$
	where $\mu$ is a large positive constant to be determined later and $\zeta(t)$ a smooth increasing function on $[0, T]$ satisfying $\zeta(0) = 0$ and $\zeta(t) = 1$ for $t \geq t_1$. 	
	By (\ref{Cap_w_t-Lw}),
	\begin{equation}\label{Z_t-LZ}
	\begin{aligned}
	Z_t - a^{kl} Z_{kl} + b^k Z_k &= \zeta^2 (W_t - a^{kl} W_{kl} + b^k W_k) + 2 \zeta \zeta_t W + \mu \\
	& \geq \zeta^2 \big(\theta^2 |\nabla^2 v|^2 - C|\nabla v|^3 - C |\nabla v|^2 - C\big) - C \zeta |W| + \mu.
	\end{aligned}
	\end{equation}
	
	Suppose $Z$ attains a minimum at $(x_0, t_0) \in \overline{M_T}$ with $Z(x_0, t_0) < 0$,
	then $W(x_0, t_0) < 0$ and $t_0 > 0$. 
	Assume $x_0 \in M$.
	Since $\kappa - Ad \geq \frac{\kappa}{2}$ and $\tilde{w} \geq 0$, then $w + \frac{\kappa}{2}\tilde{w} \leq W <0$.
	By (\ref{Harn_w&tilde_w}),
	\begin{equation}\label{Harn_ineq_5}
	\theta|\nabla v|^2 \leq \tilde{w} \leq \frac{-2w}{\kappa} \leq \frac{2\sqrt{n}}{\kappa \theta} |\nabla^2 v|,
	\end{equation}
	and hence,
	\begin{equation}\label{Harn_ineq_6}
	|W| \leq |w| + \kappa |\tilde{w}| \leq \frac{\sqrt{n}}{\theta} |\nabla^2 v| + \frac{\kappa}{\theta} |\nabla v|^2 \leq \frac{3\sqrt{n}}{\theta^3} |\nabla^2 v|,
	\end{equation}
	By Cauchy-Schwarz inequality,
	\begin{equation}\label{Harn_ineq_7}
	C\zeta |W| \leq C \zeta |\nabla^2 v| \leq \epsilon \zeta^2 |\nabla^2 v|^2 + \frac{C}{\epsilon}.
	\end{equation}
	By (\ref{Z_t-LZ}), (\ref{Harn_ineq_5}), (\ref{Harn_ineq_7}) with $\epsilon = \frac{\theta^2}{2}$, and $0 \leq \zeta \leq 1$, we have
	\begin{equation}\label{Z_t-LZ&Z<0}
	\begin{aligned}
	Z_t - a^{kl} Z_{kl} + b^k Z_k 
	&\geq \zeta^2 \Big(\frac{\theta^2}{2} |\nabla^2 v|^2 - C|\nabla v|^3 - C |\nabla v|^2\Big) - C_1' + \mu \\
	&\geq \zeta^2 \Big(\frac{\kappa^2 \theta^6}{8n} |\nabla v|^4 - C|\nabla v|^3 - C |\nabla v|^2\Big) - C_1' + \mu \\
	&\geq - C_1 - C_1' + \mu >0,
	\end{aligned}
	\end{equation}
	where $\frac{\kappa^2 \theta^6}{8n} |\nabla v|^4 - C|\nabla v|^3 - C |\nabla v|^2 \geq - C_1$ with $C_1$ independent of $|\nabla v|$, and $\mu \geq C_1 + C_1' + 1$.
	(\ref{Z_t-LZ&Z<0}) contradicts to $(x_0, t_0)$ being a minimum point of $Z$. Hence, $(x_0, t_0) \in \partial M \times (0, T]$. Since $Z < 0$ holds near $(x_0, t_0)$, then (\ref{Z_t-LZ&Z<0}) holds near $(x_0, t_0)$, and by strong maximum principle,
	\begin{equation}\label{strong_max}
	Z_\nu (x_0, t_0) > 0.
	\end{equation}
	On $\partial M$, since $v_\nu = 0$, and $v_{\nu t} = w_\nu + \tilde{w}_\nu = 0$, then 
	\begin{equation}
	\begin{aligned}
	Z_\nu &= \zeta^2 \big(w_\nu + (\kappa - Ad)\tilde{w}_\nu - A\tilde{w}\big) \\
	&= \zeta^2 \big((\kappa - Ad - 1)\tilde{w}_\nu - A\tilde{w}\big)\\
	&\leq \zeta^2 (|\tilde{w}_\nu| - A\theta |\nabla v|^2).
	\end{aligned}
	\end{equation}
	We claim that $|\tilde{w}_\nu| \leq C_2 |\nabla v|^2$ on $\partial M \times (0, T]$, where the constant $C_2$ is independent of $\delta_0$. If we choose $A$ large enough so that $A\theta \geq C_2$, then on $\partial M \times (0, T]$, we have $Z_\nu \leq 0,$
	which contradicts to (\ref{strong_max}).
	Therefore, $Z \geq 0$ in $\overline{M_T}$.
	
	Now we prove the claim above. Since $\tilde{w}_\nu = a^{ij}_\nu v_i v_j + 2a^{ij} v_{i\nu} v_j$, it suffices to prove that $|a^{ij} v_{i\nu} v_j| \leq C_2' |\nabla v|^2$.
	Let $(a_{ij})$ be the inverse matrix of $(a^{ij})$. Then $(a_{ij})$ is a Riemannian metric on $\overline{M}$ with connection $\tilde{\nabla}$ and Christoffel symbols $\tilde{\Gamma}^k_{ij}$. We choose a local orthonormal frame $\{ e_\alpha \}$ under the metric $(a_{\alpha\beta})$, i.e., $(a_{\alpha\beta}) = (\delta_{\alpha\beta})$, with $e_n = \frac{\nu}{\sqrt{a(\nu, \nu)}}$. Since $v_n = 0$ on $\partial M$, then 
	\begin{equation}
	\begin{aligned}
	a^{ij} v_j \tilde{\nabla}^2_{i\nu} v &= \sqrt{a(\nu, \nu)} \sum_{\alpha < n} v_\alpha \tilde{\nabla}^2_{\alpha n} v\\
	&= \sqrt{a(\nu, \nu)} \sum_{\alpha, \beta < n} v_\alpha [e_\alpha (v_n) - \tilde{\Gamma}^\beta_{\alpha n} v_\beta]\\
	&= - \sqrt{a(\nu, \nu)} \sum_{\alpha, \beta < n} \tilde{\Gamma}^\beta_{\alpha n} v_\alpha v_\beta\\
	&= - a^{\alpha\sigma}\tilde{\Gamma}^{\beta}_{\sigma\nu}v_\alpha v_\beta,
	\end{aligned}
	\end{equation}
	\footnote{Under the orthonoraml frame $\{e_\alpha\}$, $\tilde{\Gamma}_{\sigma\nu}^\beta = a(\nabla_\sigma \nu, e_\beta) = a\Big( \sqrt{a(\nu, \nu)}\, \nabla_\sigma e_n, e_\beta \Big) + a\Big( e_\sigma\big(\sqrt{a(\nu, \nu)}\big)\, e_n, e_\beta\Big) = \sqrt{a(\nu, \nu)}\, \tilde{\Gamma}_{\sigma n}^\beta$.}
	and
	\begin{equation}
	\begin{aligned}
	a^{ij}v_{i\nu}v_j &= a^{ij}(\nabla^2_{i\nu}v - \tilde{\nabla}^2_{i\nu}v) v_j + a^{ij} v_j \tilde{\nabla}^2_{i\nu} v \\
	& = a^{ij}(- \Gamma_{i\nu}^k + \tilde{\Gamma}_{i\nu}^k) v_k v_j + a^{ij} v_j \tilde{\nabla}^2_{i\nu} v\\
	&= - a^{ij} \Gamma_{i\nu}^k v_k v_j,
	\end{aligned}
	\end{equation}
	and hence $|a^{ij} v_{i\nu} v_j| \leq C_2' |\nabla v|^2$. Here $\Gamma_{i\nu}^k$ is bounded by $H$ and $\sup_{\partial M}||\nabla_\nu \nu||_g$. We have proved the claim above.
	
	In $\overline{M} \times [t_1, t_2]$, since $Z \geq 0$ and $\zeta = 1$, that is, $w + (\kappa - Ad)\tilde{w} + \mu t \geq 0$, then
	\begin{equation}\label{Harn_ineq_8}
	v_t = w + \tilde{w} \geq [1 - (\kappa - Ad)] \tilde{w} - \mu t \geq (1 - \frac{\kappa}{2})\theta |\nabla v|^2 - \mu t.
	\end{equation}
	
	For any $x_1, x_2 \in \overline{M}$, and any smooth path $l : [0, 1] \rightarrow \overline{M}$ with $l(0) = x_1, l(1) = x_2$, let $\eta(s) := \big(l(s), (1-s) t_1 + s t_2\big)$ be a smooth path in $\overline{M} \times [t_1, t_2]$, by (\ref{Harn_ineq_8}),
	\begin{equation}
	\begin{aligned}
	& \qquad v(x_2, t_2) - v(x_1, t_1)\\
	&= \int_{0}^{1} \frac{d}{ds} v\big(\eta(s)\big) ds \\
	&= \int_{0}^{1} g(\dot{l}, \nabla v) + (t_2 - t_1) v_t \; ds\\
	&\geq \int_{0}^{1} -|\dot{l}|\,|\nabla v| + (t_2 - t_1) \left[(1 - \frac{\kappa}{2})\theta |\nabla v|^2 - \mu \big( (1-s) t_1 + s t_2 \big) \right] \; ds \\
	&\geq - \frac{\int_{0}^{1} |\dot{l}|^2 ds}{4(1 - \frac{\kappa}{2})\theta (t_2 - t_1)} - \frac{1}{2}(t_2^2 - t_1^2)\mu.
	\end{aligned}
	\end{equation}
	Consequently,
	\begin{equation}
	\begin{aligned}
	\log \frac{u(x_2, t_2) + \epsilon_0}{u(x_1, t_1) + \epsilon_0} &\geq - \frac{r^2(x_1, x_2)}{4(1 - \frac{\kappa}{2})\theta (t_2 - t_1)} - \frac{1}{2}(t_2^2 - t_1^2)\mu\\
	&\geq - \frac{\big(\text{diag}(\overline{M})\big)^2}{4(1 - \frac{\kappa}{2})\theta (t_2 - t_1)} - \frac{1}{2}(t_2^2 - t_1^2)\mu,
	\end{aligned}
	\end{equation}
	and (\ref{Harnack's_ineq}) follows immediately when we take $\epsilon_0 \rightarrow 0$.
\end{proof}


\bigskip

\begin{thebibliography}{99}
		
	\bibitem{CNS85}
	L. A. Caffarelli, L. Nirenberg and J. Spruck, 
	{\em The Dirichlet problem for nonlinear second-order elliptic equations III: Functions of eigenvalues of the Hessians},
	{Acta Math.}  {\bf 155} (1985), 261--301.
	
	\bibitem{Cao85}
	H.-D. Cao, 
	{\em Deformation of K\"ahler metrics to K\"ahler-Einstein metrics on compact K\"ahler manifolds},
	{Invent. Math.}  {\bf 81} (1985), 359-372.
	
	\bibitem{ChenR00}
	R. Chen,
	{\em A remark on the Harnack inequality for non-self-adjoint evolution equations},
	{Proc. Amer. Math. Soc.} {\bf 129} (2000), 2163-2173.
	
	\bibitem{ChouWang01}
	K.-S. Chou and X.-J. Wang,
	{\em A variational theory of the Hessian equation},
	{Comm. Pure Applied Math.} {\bf 54} (2001), 1029-1064.
	
	\bibitem{Dong88}
	G.-C. Dong,
	{\em Initial and nonlinear oblique boundary problems for fully nonlinear parabolic equations},
	{J. Part. Diff. Equ. Series A} {\bf 1} (1988), 12-42.
	
	\bibitem{Evans82}
	L. C. Evans, 
	{\em Classical solutions of fully nonlinear, convex, second-order elliptic equations},
	{Comm. Pure Applied Math.}  {\bf 35} (1982), 333-363.
	
	\bibitem{EvansBk}
	L. C. Evans, 
	{\em Partial Differential Equations},
	Grad. Stud. in Math., vol. 19, AMS, Providence, RI, 2010.
	
	\bibitem{GT}
	D. Gilbarg and N. Trudinger,
	{\em Elliptic partial differential equations of second-order}, 
	Springer Science $\&$ Business Media, 2001.
	
	\bibitem{Guan14}
	B. Guan, 
	{\em Second-order estimates and regularity for fully nonlinear elliptic equations on Riemannian manifolds},
	{Duke Math. J.}  {\bf 163} (2014), 1492-1524. 
	
	\bibitem{Guan18}
	B. Guan,
	{\em The Dirichlet problem for fully nonlinear elliptic equations on Riemannian manifolds},
	arXiv:1403.2133.
	
	\bibitem{GuanShiSui15}
	B. Guan, S. Shi and Z. Sui,
	{\em On estimates for fully nonlinear parabolic equations on Riemannian manifolds},
	{Anal. PDE} {\bf 8} (2015), 1145-1164.
	
	\bibitem{GuanXiang18}
	B. Guan and N. Xiang,
	{\em On estimates for fully nonlinear elliptic equations with Neumann boundary conditions on Riemannian Manifolds}, arXiv:1811.12448
	
	\bibitem{GuoThesis19}
	S. Guo,
	{\em On Neumann Problems for Fully Nonlinear Elliptic and Parabolic Equations on Manifolds},
	Electronic Thesis or Dissertation. Ohio State University, 2019. https://etd.ohiolink.edu/
	
	\bibitem{HouMaWu10}
	Z. Hou, X. Ma and D. Wu,
	{\em A second-order estimate for complex Hessian equations on a compact Kähler manifold},
	{Math. Rev. Lett.} {\bf 17} (2010), 547-561.
		
	\bibitem{Kry82}
	N. V. Krylov, 
	{\em Boundedly inhomogeneous elliptic and parabolic equations},
	{Izv. Akad. Nauk SSSR Ser. Mat.}  {\bf 46} (1982), 487-523.
	
	\bibitem{LSU68}
	O.A. Lady\u{z}enskaja, V.A. Solonnikov, and N.N. Ural'ceva,
	{\em Linear and Quasi-linear Equations of Parabolic Type},
	Izd. Nauka, Moscow, 1967, [Russian] English translation: Amer. Math. Soc., Providence, R.I., 1968. 
	
	\bibitem{LiYau86}
	P. Li and S. T. Yau,
	{\em On the parabolic kernel of the Schr\"{o}dinger operator},
	{Acta Math.} {\bf 156} (1986), 153–201.
	
	\bibitem{S.Li94}
	S.-Y. Li,
	{\em On the Neumann Problems for Complex Monge-Amp\`ere Equations},
	{Indiana Univ. Math. J.} {\bf 43} (1994), 1099-1122.
	
	\bibitem{S.Li99}
	S.-Y. Li,
	{\em On the oblique boundary problems for Monge-Amp\`ere equations},
	{Pac. J. Math.} {\bf 190} (1999), 155-172.
	
	\bibitem{Lieberman96}
	G. Lieberman,
	{\em second-order parabolic differential equations},
	World Scientific, 1996.
	
	\bibitem{Lieberman13}
	G. Lieberman,
	{\em Oblique boundary Problems for Elliptic Equations},
	World Scientific, 2013.
	
	\bibitem{LiebTru86}
	G. M. Lieberman and N. S. Trudinger,
	{\em Nonlinear Oblique boundary Problems for Nonlinear Elliptic Equations},
	{Trans. Amer. Math. Soc.} {\bf 295} (1986), 509-546.
	
	\bibitem{LionsTrudUrbas86}
	P.‐L. Lions, N. S. Trudinger and  J. Urbas,
	{\em The Neumann problem for equations of Monge‐Amp\`ere type},
	{Comm. Pure Appl. Math.} {\bf 39} (1986), 539-563.
	
	\bibitem{MaQiu19}
	X. Ma and G. Qiu, 
	{\em The Neumann problem for Hessian equations},
	{Comm. Math. Phys.}  {\bf 366} (2019), 1-28.
	
	\bibitem{PhongTo17}
	D. H. Phong and D. T. T\^o,
	{\em Fully non-linear parabolic equations on compact Hermitian manifolds},
	arXiv:1711.10697.
	
	\bibitem{SchnurerSmoczyk03}
	O. C. Schnürer and K. Smoczyk,
	{\em Neumann and second boundary problems for Hessian and Gauß curvature flows},
	{Ann. Inst. H. Poincar\'e Anal. Non Lin\'eaire}, {\bf 20} (2003), 1043-1073.
	
	\bibitem{Sze15}
	G. Sz\'ekelyhidi, 
	{\em Fully non-linear elliptic equations on compact Hermitian manifolds}, 
	{J. Diff. Geom.} {\bf 109} (2018), 337-378.
	
	\bibitem{Tru87}
	N. Trudinger,
	{\em On degenerate fully nonlinear elliptic equations in balls},
	{Bull. Australian Math. Soc.} {\bf 35} (1987), 299-307.
	
	\bibitem{Ural'tseva91}
	N. N. Ural'tseva,
	{\em A nonlinear problem with an oblique derivative for parabolic equations},
	{ Zapiski Nauchnykh Seminarov Leningradskogo Otdeleniya Matematicheskogo Instituta im. V. A. Steklova Akademiya Nauk SSSR} {\bf 188} (1991), 143-158. Translated in: {J. Math. Sciences} {\bf 70} (1994), 1817–1827.
	
	\bibitem{Urbas95}
	J. Urbas, 
	{\em Nonlinear oblique boundary problems for Hessian equations in two dimensions},
	{Ann. Inst. H. Poincar\'e Anal. Non Lin\'eaire}  {\bf 12} (1995), 507-575.
	
	\bibitem{Urbas96}
	J. Urbas, 
	{\em Nonlinear oblique boundary problems for two-dimensional curvature equations},
	{Adv. Diff. Equ.}  {\bf 1} (1996), 301-336.
	
	\bibitem{Urbas98}
	J. Urbas, 
	{\em Oblique boundary problems for equations of Monge-Amp\`ere type},
	{Calc. Var. Partial Diff. Equ.}  {\bf 7} (1998), 19-39.
	
	\bibitem{Wang92}
	X.-J. Wang,
	{\em Oblique derivative problems for the equations of Monge-Amp\`ere type},
	{Chinese J. Contemp. Math.} {\bf 13} (1992), 13-22.
	
	\bibitem{Yau95}
	S.-T. Yau, 
	{\em Harnack inequality for non-self-adjoint evolution equations},
	{Math. Res. Lett.}  {\bf 2} (1995), 387–399.
	
	
\end{thebibliography}
\end{document}